\documentclass[11pt]{article}
\usepackage{fullpage}
\usepackage{amsmath,amsthm,amssymb}
\usepackage{bm}
\usepackage{color}
\usepackage{graphicx}
\usepackage{natbib}
\usepackage{caption}
\usepackage{subcaption}
\usepackage{float}
\usepackage{algorithm,setspace}
\usepackage[noend]{algpseudocode}
\usepackage{listings}
\usepackage{standalone}
\usepackage{enumitem}
\usepackage[textwidth=2cm,textsize=footnotesize]{todonotes}
\usepackage{hyperref}
\hypersetup{
    colorlinks=true,
    linkcolor=red,
    citecolor=blue,
    filecolor=magenta,
    urlcolor=cyan,
}

\newcommand{\llangle}{\left\langle}
\newcommand{\rrangle}{\right\rangle}
\newcommand{\E}{\mathbf{E}}

\newcommand{\Var}{\textup{Var}}
\newcommand{\R}{\mathbf{R}}

\newcommand{\ones}{\mathbf{1}}
\newcommand{\sign}{\textup{sign}}

\newcommand{\Tr}{\textup{Tr}}

\newcommand{\vertiii}[1]{{\bigl\vert\kern-0.25ex\bigl\vert\kern-0.25ex\bigl\vert #1 \bigr\vert\kern-0.25ex\bigr\vert\kern-0.25ex\bigr\vert}}
\newcommand{\norm}[1] {\left \| #1 \right \|}

\newcommand{\var}{\textup{Var}}

\newcommand{\cF}{\mathcal{F}}
\newcommand{\cO}{\mathcal{O}}

\newcommand{\hG}{\hat{G}}
\newcommand{\hM}{\hat{M}}

\newcommand{\hW}{\hat{W}}

\newcommand{\trace}{\textup{Tr}}
\newcommand{\polar}{\textup{Orth}}
\newcommand{\ortho}{\textup{Orth}}

\newtheorem{theorem}{Theorem}[section]
\newtheorem{lemma}{Lemma}[section]

\newtheorem{assumption}{Assumption}

\theoremstyle{definition}
\newtheorem{example}{Example}[section]
\newtheorem{remark}{Remark}[section]

\title{Second-Moment Stochastic Approximation Methods}
\author{Tao Jiang$^*$ \and Lin Xiao\footnote{
\,Fundamental AI Research (FAIR), Meta Superintelligence Labs
(\texttt{taojiang@meta.com}, \texttt{linx@meta.com}).}
}
\date{September 28, 2026}

\begin{document}
\maketitle

\begin{abstract}
Classical stochastic approximation methods rely on estimators of the first moment (mean) of a random regression function.
We study methods that employ estimators of both the first and the second moments, which include modern deep-learning optimizers such as Adam and Muon as special cases.
We derive second-moment stochastic approximation methods through the lens of optimal preconditioning for solving matrix equations, and develop a two-stage framework for their convergence analysis.
The first stage focuses on the analysis of conceptual (impractical) methods that rely on the exact first and second moments.
In the second stage, we replace the exact moments with their respective estimators, and invoke Dvoretzky's theorem to show that the resulting practical methods converge almost surely to a neighborhood of the target solution.
The size of the neighborhood depends on the biases and variances of the first- and second-moment estimators.
We derive concrete bounds for Muon and a spectral variant of Adam that determine the radius of their neighborhood of convergence.
\end{abstract}

\section{Introduction}
Consider the problem of solving a matrix equation $M(X)=0$ where $M$ is a mapping from $\R^{m\times n}$ (space of $m$ by~$n$ matrices) to itself. We focus on the stochastic setting
\begin{equation}\label{eqn:M-def}
M(X) = \E_\xi [G(X, \xi)],
\end{equation}
where $G$ is a random regression function and $\xi$ is a random variable with some unknown distribution.
This is equivalent to solving a vector equation $\mathcal{M}(x)=0$ where $x\in\R^{mn}$ is a vector of length~$mn$ and $\mathcal{M}:\R^{mn}\to\R^{mn}$ is the vectorization of~$M$.
However, the matrix form arises naturally from many applications and allows us to use matrix factorization methods such as the singular value decomposition (SVD) to construct better preconditioners, which is the focus of this paper.

Our motivation comes from studying algorithms for training large-scale deep learning models \cite[e.g.,][]{Goodfellow2016book,Bishop2024DeepLearning}. In such applications, one need to minimize a loss function of the form $f(X)=\E_\xi[F(X,\xi)]$ where~$X$ denotes the model parameters, $\xi$ represents random samples from a dataset, and~$F(\cdot,\xi):\R^{m\times n}\to\R$ is the per-sample loss function.
In general, the model parameters consist of a large set of matrices of various sizes. To simplify presentation, we focus on the setting of a single matrix variable, but the methods we study can be applied block-wise to every matrix in a more complex model.
In this context, we have $G(X,\xi)=\nabla F(X,\xi)$ and $M(X)=\nabla f(X)$, thus solving the matrix equation $M(X)=0$ corresponds to finding a stationary point of~$f$.

There is a large literature on solving systems of nonlinear equations \cite[e.g.,][]{OrtegaRheinboldt2000book,Kelly2022Nonlinear}, and most methods require accurate evaluation of $M(\cdot)$ and for Newton's method also its Jacobian \cite{Kelley2003Newton}.
None of them can be applied to the stochastic setting where we only have access to the random map $G(\cdot,\xi)$.
In this setting, we resort to the stochastic approximation (SA) method of Robbins and Monro \cite{RobbinsMonro51}:
\begin{equation}\label{eqn:sa}
    X_{t+1} = X_t - \alpha_t G(X_t,\xi_t),
\end{equation}
where $\xi_t$ is a realization of the random variable~$\xi$ at iteration~$t$ and $\alpha_t>0$ is an appropriately chosen step size.
Since $G(X_t,\xi_t)$ is a random estimator of $M(X_t)$, we can interpret~\eqref{eqn:sa} as a practical variant of the conceptual method $X_{t+1} = X_t - \alpha_t M(X_t)$.
In general, we can employ a more sophisticated estimator of $M(X_t)$, e.g., the exponential moving average (EMA) of $G(X_t,\xi_t)$,
\begin{equation}\label{eqn:ema}
\hM_t = \beta_t \hM_{t-1} + (1-\beta_t) G(X_t,\xi_t),
\end{equation}
where $\beta_t\in[0,1)$ is a smoothing factor.
This leads to the SA method
\[
X_{t+1} = X_t - \alpha_t \hM_t.
\]
The use of EMA, also known as the \emph{momentum}, was analyzed for stochastic convex optimization as early as in \cite{GupalBazhenov72}, and became very popular in deep learning following the work of Sutskever et al.~\cite{Sutskever2013momentum}.
In the rest of this paper, we use $\hM_t$ to denote any estimator of $M(X_t)$, including $G(X_t,\xi_t)$ itself and its EMA as special cases.
Since $\hM_t$ is an estimate of the first moment $M(X_t)$, we call such algorithms \emph{first-moment} SA methods.

In this paper, we study stochastic approximation methods that also make use of the \emph{second moment} of the random mapping~$G(X,\xi)$.
The rationale behind using the second moment in SA parallels that of using second-order derivatives in optimization.
In deterministic optimization, we use second-order derivatives (Hessian of the loss function, corresponding to Jacobian of~$M$ in our context) to construct preconditioners to combat anisotropic curvature.
In stochastic approximation, the presence of randomness/noise renders the second-order derivatives hard to evaluate and less effective.
On the other hand, the second moment of~$G(X,\xi)$ carries statistical information that can be exploited to better combat randomness/noise.
Specifically, we can use online estimators of the second moment to construct preconditioned SA methods.

\subsection{Second-moment stochastic approximation}
\label{sec:2nd-moment-sa}

There are several different ways to define the second moment for the matrix-valued map~$G(X,\xi)$, depending on the type of matrix product used. The three apparent choices are
\begin{subequations}
\label{eqn:2nd-moment}
\begin{align}
\E_\xi\bigl[G(X,\xi)\odot G(X,\xi)\bigr] &\in \R^{m\times n},
\label{eqn:2nd-moment-elem}\\
\E_\xi\bigl[G(X,\xi) G(X,\xi)^T\bigr] &\in \R^{m\times m},
\label{eqn:2nd-moment-GGT} \\
\mbox{or}\quad\E_\xi\bigl[G(X,\xi)^T G(X,\xi)\bigr] &\in \R^{n\times n},
\label{eqn:2nd-moment-GTG}
\end{align}
\end{subequations}
where the first one uses the Hadamard (element-wise) product, denoted as $\odot$, and the second and third definitions use matrix multiplication in different orders.
For convenience, we use $W(X)$ to denote the second moment of $G(X,\xi)$ regardless of the specific definition, which should be self-evident from the context.
To cope with the ambiguity in the size of~$W$ under different definitions, we use $\dim(W)$ to denote its dimension.
Specifically, we have $\dim(W)=m\times n$, $m\times m$ and $n\times n$, respectively, for the three cases listed above.

Conceptually, a general second-moment method takes the form
\begin{equation}\label{eqn:conceptual-smsa}
X_{t+1} = X_t - \alpha_t \phi\bigl(M(X_t), W(X_t)\bigr),
\end{equation}
where $\phi:\R^{m\times n}\times\R^{\dim(W)}\to\R^{m\times n}$ is a \emph{blending function} that combines the first and second moments to form the update direction.
Since $M(X_t)$ and $W(X_t)$ cannot be evaluated exactly in practice, we need to replace them with estimators $\hM_t$ and $\hW_t$ respectively, leading to the \emph{practical} second-moment SA method
\begin{equation}\label{eqn:practical-smsa}
X_{t+1} = X_t - \alpha_t \phi\bigl(\hM_t,\hW_t\bigr) .
\end{equation}
Below we give several concrete examples of such methods.
For the first moment, they all use the EMA estimator $\hM_t$ in~\eqref{eqn:ema} with a constant smoothing factor~$\beta_1$, which we do not repeat in the descriptions below.
To simplify the presentation, we use the notation $G_t$ for $G(X_t,\xi_t)$.
\begin{itemize}
\item Adam \cite{KingmaBa2014adam} can be interpreted through the first definition of the second moment~\eqref{eqn:2nd-moment-elem}. It uses the EMA of $G_t\odot G_t$, with a smoothing factor~$\beta_2$, as the estimator for $W(X_t)$, i.e.,
\[
\hW_t = \beta_2 \hW_{t-1} + (1-\beta_2) G_t\odot G_t.
\]
Given the estimates $\hM_t$ and $\hW_t$, and a prespecified $\epsilon>0$, Adam follows the update rule
\begin{equation}\label{eqn:adam}
\begin{aligned}
X_{t+1} &= X_t - \alpha_t \hM_t \oslash \textstyle\sqrt{\hW_t + \epsilon},
\end{aligned}
\end{equation}
where $\oslash$ denotes element-wise division of two matrices and the square-root is taken element-wise.
In this case, the blending function $\phi(M,W)=M\oslash\sqrt{W+\epsilon}$ for some $\epsilon>0$.

\item The RMSProp method \cite{rmsprop2012lecture} is a precursor to Adam, which uses the simple estimator $\hM_t=G_t$ but the same EMA estimator for $\hW_t$ and the same blending function~$\phi$ as for Adam.
\item ASGO (Adaptive Structured Gradient Optimization) \cite{an2025asgo} can be interpreted through the definition of $W(X)$ under matrix multiplication in~\eqref{eqn:2nd-moment-GGT} if $m\leq n$ and \eqref{eqn:2nd-moment-GTG} if $m>n$.
Specifically, for $m\leq n$, it approximates the second moment $W(X_t)$ with the EMA of $G_t G_t^T$ and adopts a matrix variant of Adam's update rule:
\begin{equation}\label{eqn:asgo}
\begin{aligned}
\hW_t & = \beta_2 \hW_{t-1} + (1-\beta_2) G_t G_t^T, \\
X_{t+1} &= X_t - \alpha_t (\hW_t + \epsilon I)^{-1/2} \hM_t.
\end{aligned}
\end{equation}
Corresponding to the general framework~\eqref{eqn:practical-smsa}, it has $\phi(M,W)=(W+\epsilon I)^{-1/2}M$.

\item Muon \cite{jordan2024muon} employs the update rule
\begin{equation}\label{eqn:muon-newtonschulz}
X_{t+1} = X_t - \alpha_t\, \texttt{NewtonSchulz5}(\hM_t),
\end{equation}
where $\texttt{NewtonSchulz5}(\cdot)$ is a carefully tuned Newton-Schulz iterative method \cite{higham2008,Kovarik1970,BjorckBowie1971,bernstein2018signsgd} to approximate the nearest semi-orthogonal matrix to~$\hM_t$, denoted as $\ortho(\hM_t)$.
Suppose the SVD of~$\hM_t$ is $\hM_t=U_t \Sigma_t V_t^T$ where $U_t\in\R^{m\times r}$ and $V_t\in\R^{n\times r}$  (with $r=\min\{m,n\}$) have orthonormal columns, it holds that $\ortho(\hM_t)=U_t V_t^T$.
We focus on the \emph{ideal} Muon method
\[
X_{t+1} = X_t - \alpha_t\, \ortho(\hM_t),
\]
where the blending function $\phi(M, W)=\ortho(M)$.
Notice that Muon does not use a second-moment estimator.
Nevertheless, we will show that it can be derived from our framework of second-moment preconditioning, even though the simplified form, under favorable noise structure, does not involve the second moment explicitly.
\end{itemize}
Many other modern deep-learning optimizers can also be cast into the same framework, including SignSGD~\cite{fabian1960stochastic}, Signum~\cite{bernstein2018signsgd}, Lion~\cite{chen2023symbolic}, AdaGrad~\cite{DuchiHazanSinger2011adagrad}, Nadam~\cite{dozat2016incorporating}, Adam-mini~\cite{zhang2024adam}, AdEMAMix~\cite{pagliardini2024ademamix}, Shampoo~\cite{gupta2018shampoo} and SOAP~\cite{vyas2025soap}.
We omit the details here due to space limitations.

Despite the apparent differences between the forms of these methods, we show that most of them can be derived or interpreted from a unified framework of second-moment preconditioning. The details are given in Section~\ref{sec:preconditioning}.

\subsection{A two-stage framework for convergence analysis}
\label{sec:two-stage-intro}

Accompanying our general formulation of second-moment SA methods, we develop a two-stage framework for their convergence analysis.

\begin{itemize}
\item \textbf{Stage~1}. We analyze the conceptual algorithm~\eqref{eqn:conceptual-smsa} under the following pair of assumptions:
\begin{equation}\label{eqn:conceptual-phi-assumptions}
\begin{aligned}
\llangle \phi\bigl(M(X), W(X)\bigr),\, X-X_*\rrangle & \geq \mu \|X-X_*\|_F^2, \\
\left\|\phi\bigl(M(X), W(X)\bigr)\right\|_F &\leq L\|X-X_*\|_F + C ,
\end{aligned}
\end{equation}
where $X_*$ is a solution to $M(X)=0$, $\|\cdot\|_F$ denotes the matrix Frobenius norm, and $\mu$, $L$ and $C$ are nonnegative constants.
The first inequality
is an extension of the original assumption by Robbins and Monro \cite{RobbinsMonro51}, which played a fundamental role in the classical SA literature \cite[see, e.g.][]{Wasan69book}.
The second inequality, which bounds the magnitude of~$\phi$, automatically holds for many second-moment methods including Adam and Muon.
Therefore, our main focus is on the first inequality, which we call the ``aiming'' condition.
\item \textbf{Stage~2}.
To analyze the practical method~\eqref{eqn:practical-smsa}, we need an additional assumption on how the update $\phi(\hM_t,\hW_t)$ deviates from its nominal value $\phi(M(X_t), W(X_t))$. Specifically, we assume
\begin{equation}\label{eqn:practical-conceptual-diff}
\epsilon_t:= \bigl\|\E\bigl[\phi(\hM_t,\hW_t)\bigr] - \phi\bigl(M(X_t), W(X_t)\bigr)\bigr\|_F
\end{equation}
is bounded (having finite limit superior) almost surely.
This assumption allows us to invoke Dvoretzky’s theorem \cite{Dvoretzky1956,DermanSacks1959,Venter1966} to show that the practical method converges almost surely to a neighborhood of $X_*$, and the size of the neighborhood depends on the asymptotic upper bound on~$\epsilon_t$.
Then we derive concrete bounds on $\epsilon_t$ in the specific cases of Muon and ASGO, which fittingly depend on the bias and variance of both estimators $\hM_t$ and $\hW_t$.
\end{itemize}

The two-stage procedure outlined above provides a general framework that can unify the analysis of many existing algorithms as well as guide the development of new ones.
For example, RMSProp and Adam share the same blending function~$\phi$ and second-moment estimator but differ in their first-moment estimators.
In particular, RMSProp uses $G_t$ as its first-moment estimator which has no bias but high variance, while Adam uses an EMA of $G_t$ which has bias but low variance.
Therefore, they share the same analysis except that the bias and variance of their respective first-moment estimators lead to different bounds on~$\epsilon_t$ in~\eqref{eqn:practical-conceptual-diff}, which dictate the size of the neighborhood around $X_*$ they converge to.

As an example of developing new variants, AdEMAMix \cite{pagliardini2024ademamix} combines two EMAs of $G_t$ with different smoothing factors as the first-moment estimator. We can also employ more sophisticated estimators, such as double exponential moving average \cite[e.g.,][]{Mulloy1994} of $G_t$ as the first-moment estimator and that of $G_t\odot G_t$ or $G_t G_t^T$ as the second-moment estimator.
This gives rise to a plethora of algorithms that would be challenging to analyze separately one by one.
With our framework, we first identify the underlying blending function~$\phi$ and check if it satisfies the assumptions in~\eqref{eqn:conceptual-phi-assumptions}, and then characterize how the bias and variance of the first- and second-moment estimators translate into a bound on~$\epsilon_t$ defined in~\eqref{eqn:practical-conceptual-diff}.

\subsection{Outline}
In Section~\ref{sec:preconditioning}, we present a general framework of second-moment preconditioning for solving stochastic matrix equations and derive the blending functions for Muon, ASGO and Adam.
In Section~\ref{sec:two-stage-framework}, we present details of the two-stage framework for convergence analysis outlined in Section~\ref{sec:two-stage-intro}. Our analysis is based on a slightly more general form of~\eqref{eqn:conceptual-smsa} that also include decoupled weight decay~\cite{Loshchilov2019AdamW}, which is an essential ingredient for practical methods. Our theory indeed supports its necessity.
In Sections~\ref{sec:muon-analysis} and Section~\ref{sec:asgo-analysis}, we derive concrete bounds for the deviation $\epsilon_t$ defined in~\eqref{eqn:practical-conceptual-diff} for Muon and ASGO respectively, and show that they depend on the bias and variance of the first- and second-moment estimators.
In Section~\ref{sec:experiments}, we present numerical experiments on training large language models with Muon and ASGO to illustrates some implications of our theory.
Finally, we conclude in Section~\ref{sec:conclusion} with comments on related work and the limitations of our approach.

\section{Second-moment preconditioning}
\label{sec:preconditioning}

For solving the matrix equation $M(X)=0$ where $M$ maps $\R^{m\times n}$ to itself, a full preconditioner based on its Jacobian (Newton's method) requires an $m \times n \times m \times n$ tensor, which is very expensive to compute and store for even moderate sizes of~$m$ and~$n$.
In the stochastic case with $M(X)=\E_\xi[G(X,\xi)]$,  we present a framework of \emph{second-moment} preconditioning that does not involve the Jacobian and employs much smaller preconditioners.

Different matrix products lead to different definitions of the second moment listed in~\eqref{eqn:2nd-moment}.
First, with the matrix Hadamard product, we have the preconditioned update rule
\begin{equation}\label{eqn:hadamard-psa}
X_{t+1} = X_t - P_t \odot G_t,
\end{equation}
where $P_t\in\R^{m\times n}$ is an elementwise or ``diagonal'' preconditioner. This leads to the second-moment definition in~\eqref{eqn:2nd-moment-elem}.
With left matrix multiplication, we have
\begin{equation}\label{eqn:left-mat-psa}
X_{t+1} = X_t - P_t G_t,
\end{equation}
where $P_t\in\R^{m\times m}$ can be much smaller in size than $X_t$ if $m<n$. In the case of $m>n$, we can switch the order of matrix multiplication, i.e., work with the update rule
\begin{equation}\label{eqn:right-mat-psa}
X_{t+1}=X_t-G_t P_t,
\end{equation}
where $P_t\in\R^{n\times n}$.
These two cases lead to the definitions of the second moment in~\eqref{eqn:2nd-moment-GGT} and~\eqref{eqn:2nd-moment-GTG} respectively.
Since $P_t$ may have a smaller size than $X_t$ itself, we call it a \emph{fused} preconditioner.

With any of the update rules discussed above, we want to find the optimal preconditioner~$P_t$ that can move $X_{t+1}$ as fast as possible towards some $X_*$ satisfying $M(X_*)=0$.
In the stochastic setting, we minimize the expected distance
\begin{equation}\label{eqn:expected-dist}
\E_t\!\left[\|X_{t+1}-X_*\|_F^2\right],
\end{equation}
where $\E_t[\cdot]$ denotes the conditional expectation
\begin{equation}\label{eqn:cond-expect}
\E\!\left[\,\cdot\,\big|\, X_0, \xi_0, X_1, \ldots,\xi_{t-1}, X_t \right].
\end{equation}
By substituting~\eqref{eqn:hadamard-psa}, \eqref{eqn:left-mat-psa} or~\eqref{eqn:right-mat-psa} into~\eqref{eqn:expected-dist}, we notice that the expected distance is a convex quadratic function of $P_t$, whose coefficients consist of the first and second moments of $G_t$, namely $M(X_t)$ and $W(X_t)$, as well as the unknown quantity $X_t-X_*$.
Therefore we can express the optimal preconditioner in terms of these quantities, albeit not computable.
Next we can approximate the optimal preconditioner by removing the explicit dependence on $X_t-X_*$ and absorbing its magnitude with a tunable step size $\alpha_t$ (which decreases to zero as~$X_t$ approaches~$X_*$).
This leads to a conceptual algorithm of the form~\eqref{eqn:conceptual-smsa}.
Then we replace $M(X_t)$ and $W(X_t)$ with their estimators $\hM_t$ and $\hW_t$ respectively to obtain practical second-moment SA methods.

In Section~\ref{sec:fused-precond}, we derive a conceptual second-moment method based on the fused preconditioner in~\eqref{eqn:left-mat-psa}. Then we make further simplifications that lead to the update rules of Muon and ASGO in Sections~\ref{sec:derive-muon} and~\ref{sec:derive-asgo} respectively.
Derivations for~\eqref{eqn:right-mat-psa} mirror those for~\eqref{eqn:left-mat-psa} and thus are omitted.
In Section~\ref{sec:elem-precond}, we present results for the elementwise preconditioner~\eqref{eqn:hadamard-psa}, which lead to the update rule for Adam and several variants.

\subsection{The optimal fused preconditioner and its approximation}
\label{sec:fused-precond}
In this section, we focus on the case $m\leq n$ and the preconditioned SA method~\eqref{eqn:left-mat-psa}.
Assuming that there exists an $X_*$ satisfying $M(X_*)=0$, we would like to choose~$P_t$ to minimize the expected distance from $X_{t+1}$ to $X_*$.
First, we use the update rule~\eqref{eqn:left-mat-psa} to obtain
\begin{align*}
\|X_{t+1}-X_*\|_F^2
&=\|X_t - P_t G_t -X_*\|_F^2 \\
&= \|X_t - X_*\|_F^2 - 2 \Tr\bigl((X_t-X_*)^T P_t G_t\bigr) + \Tr\bigl(P_t G_t G_t^T P_t^T\bigr).
\end{align*}
Taking the conditional expectation defined in~\eqref{eqn:cond-expect} yields
\begin{align*}
\E_{t}\bigl[\norm{X_{t+1}-X_\ast}_F^2\bigr]
&= \norm{X_t-X_\ast}_F^2 - 2 \Tr \left((X_t-X_\ast)^T P_t \E_t[G_t] \right) + \Tr \left(P_t \E_t[G_t G_t^T]P_t^T  \right) \\
&= \norm{X_t-X_\ast}_F^2 - 2 \Tr \left((X_t-X_\ast)^t P_t M_t \right) + \Tr \left(P_t W_t P_t^T  \right),
\end{align*}
where we used the linearity of expectation and replaced $\E_t[G_t]$ and $\E_t[G_t G_t^T]$ with the simpler notations $M_t$ and $W_t$ respectively.
Clearly this is a convex quadratic function of~$P_t$.

Assume that $W_t$ is positive definite. Then the expected distance is strictly convex in~$P_t$ and has a unique minimizer $P_t^\star$.
Setting the gradient with respect to $P_t$ equal to zero, i.e.,
\[
-2(X_t-X_\ast)M_t^T+2P_tW_t=0,
\]
we obtain the unique optimal preconditioner
\begin{equation}\label{eq:optimal_mat_precond}
P_t^\star=(X_t-X_\ast)M_t^TW_t^{-1}.
\end{equation}
An important property of $P_t^\star$ is that the resulting update $P_t^\star G_t$ is \emph{scaling-invariant}, meaning that it does not change if $G_t$ is multiplied by any nonzero scalar.

We must approximate $P_t^\star$ because $X_\ast$ is unknown. A simple idea is to replace the matrix $X_t-X_\ast$ with a scalar step size $\alpha_t \approx \norm{X_t-X_\ast}$ to account for its magnitude. However, the remaining part $M_t^T W_t^{-1}$ is a $n\times m$ matrix, which does not match $P_t$'s size of $m\times m$. Therefore, we need a more nuanced approximation.
For this purpose, we assume that $M_t$ has full row rank and its SVD is
\begin{equation}\label{eqn:M-svd}
M_t = U_t\Sigma_t V_t^T,
\end{equation}
where $U_t \in \R^{m \times m}$ and $V_t \in \R^{n \times m}$ have orthonormal columns and $\Sigma_t \in \R^{m \times m}_{+}$ is the diagonal matrix of singular values. Then $M_tM_t^T$ is positive definite and the SVD in~\eqref{eqn:M-svd} gives
\[
M_t^T(M_tM_t^T)^{-1/2}=V_tU_t^T.
\]
Consequently,
\begin{align*}
P_t^\star &=(X_t-X_\ast)M_t^T(M_tM_t^T)^{-1/2} (M_tM_t^T)^{1/2}W_t^{-1}\\
&=(X_t-X_\ast)V_tU_t^T(M_tM_t^T)^{1/2}W_t^{-1}.
\end{align*}
Notice that multiplication by $V_t U_t^T$ is a partial isometry, satisfying
\[
\|(X_t-X_*)V_tU_t^T\|_F=\|(X_t-X_*)V_t\|_F
\leq \|X_t-X_*\|_F,
\]
with equality when $m=n$.
Therefore we approximate $P_t^\star$ with
\[
P_t = \alpha_t (M_t M_t^T)^{1/2} W_t^{-1} ,
\]
where $\alpha_t$ is a tunable step size, roughly proportional to $\|(X_t-X_*)V_tU_t^T\|_F$.
Another benefit of the above approximation is that the overall update,
\[
P_t G_t = \alpha_t (M_t M_t^T)^{1/2} W_t^{-1} G_t,
\]
retains the scaling-invariant property of $P_t^\star G_t$ (against positive scaling of $G_t$). This is not the case, for example, by simply choosing $P_t = \alpha_t W_t^{-1}$.

For construction of a conceptual algorithm, we are most interested in its average behavior.
Therefore, we further replace $G_t$ with its expectation $M_t$ to arrive at
\begin{equation}\label{eqn:general-conceptual-method}
X_{t+1} = X_t -  \alpha_t (M_t M_t^T)^{1/2} W_t^{-1} M_t,
\end{equation}
which belongs to the general class of~\eqref{eqn:conceptual-smsa}.
We can immediately instantiate this conceptual algorithm by replacing $M_t$ and $W_t$ with their estimators $\hM_t$ and $\hW_t$ respectively:
\[
X_{t+1} = X_t - \alpha_t (\hM_t\hM_t^T)^{1/2}\hW_t^{-1}\hM_t,
\]
for example, using a pair of EMA estimators
\begin{equation}\label{eqn:ema-for-MW}
\begin{aligned}
\hM_t &= \beta_1\hM_{t-1} + (1-\beta_1) G_t, \\
\hW_t &= \beta_2\hW_{t-1} + (1-\beta_2) G_t G_t^T .
\end{aligned}
\end{equation}
However, each step of this algorithm requires four matrix multiplications, a matrix square-root and a matrix inverse, which is computationally very expensive.
Therefore we propose to further simplify it.
In the following, we present two approaches that differ in our assumptions on the noise structure in~$G_t$, which lead to Muon and ASGO respectively.

\subsection{Derivation of Muon and its variants}
\label{sec:derive-muon}
As discussed in Section~\ref{sec:fused-precond}, an important property of the optimal update $P_t^\star G_t$ or $P_t^\star M_t$ is its scaling invariance, i.e., it does not change if we multiply $G_t$ by an arbitrary positive scalar.
We would like to retain this property with any further simplification.

Given the stochastic mapping $G_t:=G(X_t,\xi_t)$, we interpret $M_t:=\E_t[G_t]$ as the signal and $G_t-M_t$ as the noise. Then we have the following decomposition of the second moment into signal and noise powers:
\begin{equation}\label{eqn:signal-noise-sum}
W_t := \E_t[G_t G_t^T] = M_t M_t^T + \E_t[(G_t-M_t) (G_t-M_t)^T].
\end{equation}
Again, assume that $W_t$ is positive definite.
To measure the relative strength of the signal and noise, we define a \emph{signal fraction} matrix
\begin{equation}\label{eqn:signal-fraction}
S_t := (M_t M_t^T)^{1/2} W_t^{-1} (M_t M_t^T)^{1/2}.
\end{equation}
Since $W_t-M_t M_t^T$ is positive semidefinite, the signal fraction matrix $S_t$ is positive semidefinite and all its eigenvalues lie in the interval $[0,1]$.
If in addition $M_t$ has full row rank, then $S_t$ is positive definite and its eigenvalues lie in $(0,1]$.
Eigenvalues near zero or one correspond to low or high signal-to-noise ratios along the direction of their associated eigenvectors.

With the definition of $S_t$, we can write the update in the conceptual algorithm~\eqref{eqn:general-conceptual-method} as
\begin{equation}\label{eqn:approx-sif-update}
\alpha_t (M_t M_t^T)^{1/2} W_t^{-1} M_t = \alpha_t S_t (M_t M_t^T)^{-1/2} M_t .
\end{equation}
If the signal fractions are relatively uniform along different directions, i.e., $S_t\approx \gamma_t I_m$ for some $\gamma_t\in(0,1]$, then we can approximate the update as $\alpha_t\gamma_t (M_t M_t^T)^{-1/2}M_t$ and assimilate $\gamma_t$ into the tunable step size $\alpha_t$.
This leads to the following simplified algorithm
\begin{equation}\label{eqn:polar-conceptual}
X_{t+1} = X_t - \alpha_t (M_t M_t^T)^{-1/2} M_t .
\end{equation}
This method retains the scaling-invariance property.
Moreover, in light of the SVD in~\eqref{eqn:M-svd}, we have $(M_t M_t^T)^{-1/2}M_t = U_t V_t^T=\ortho(M_t)$.
Therefore, it can be rewritten as
\begin{equation}\label{eqn:ortho-conceptual}
X_{t+1} = X_t - \alpha_t \ortho(M_t).
\end{equation}
In practice, we replace $M_t$ with an online estimator.
Below are a few concrete examples.

\begin{itemize}
\item The stochastic spectral descent (SSD) method \cite{carlson2015stochastic} uses $G_t$ as an estimator for $M_t$ in~\eqref{eqn:ortho-conceptual}.
\item As discussed in Section~\ref{sec:2nd-moment-sa}, Muon \cite{jordan2024muon} uses an EMA of $G_t$ as $\hM_t$.
The authors of \cite{jordan2024muon} also proposed a variant that uses a combination of the EMA and $G_t$ as an estimator, i.e.,
\begin{equation}\label{eqn:qhm}
\begin{aligned}
\hM_{t} &= \beta \hM_{t-1} + (1-\beta) G_t, \\
X_{t+1} &= X_t - \alpha_t \ortho\bigl(\gamma_t \hM_t + (1-\gamma_t) G_t\bigr),
\end{aligned}
\end{equation}
where $\beta,\gamma_t\in[0, 1]$.
This estimator is also called a quasi-hyperbolic momentum \cite{Ma2018QuasihyperbolicMA}.
\item MARS~\cite{yuan2024mars} first computes an estimator $\hG_t$ using a variance-reduction technique \cite[e.g.,][]{cutkosky2019momentum} and then uses the EMA of~$\hG_t$ to form an estimator for$M_t$ in~\eqref{eqn:ortho-conceptual}.
\end{itemize}

\subsection{Derivation of ASGO and its variants}
\label{sec:derive-asgo}
The class of algorithms presented in Section~\ref{sec:derive-muon} has the advantage of not requiring an estimator for the second moment explicitly.
But their derivation relies on the assumption that $S_t\approx \gamma_t I_m$ for some $\gamma_t\in(0,1]$, where $S_t$ is the signal fraction matrix defined in~\eqref{eqn:signal-fraction}.
When this assumption does not hold, we need to derive a simplified version of~\eqref{eqn:approx-sif-update} that captures the structure of~$S_t$.

For this purpose, we define the following matrix
\[
R_t=(M_t M_t^T)^{1/2}W_t^{-1/2},
\]
which satisfies $S_t=R_t R_t^T$.
Notice that $R_t$ is not the matrix square-root of $S_t$, but the decomposition $S_t=R_t R_t^T$ is unique up to unitary transformations.
Then we can write~\eqref{eqn:approx-sif-update} as
\begin{align*}
\alpha_t (M_t M_t^T)^{1/2}W_t^{-1}M_t
&=\alpha_t R_t W_t^{-1/2} M_t.
\end{align*}
Instead of dropping the full signal-fraction matrix $S_t$ (approximating it with a scaled identity matrix), here we drop $R_t$, which is ``half'' of $S_t$, and retain the other half $R_t^T$ within $W_t^{-1/2}M_t$.
This leads to the following simplified conceptual algorithm:
\begin{equation}\label{eqn:sqrt-inv-conceptual}
X_{t+1} = X_t - \alpha_t W_t^{-1/2} M_t.
\end{equation}
Compared with~\eqref{eqn:polar-conceptual}, we replaced $\bigl(M_t M_t^T\bigr)^{-1/2}$ with $W_t^{-1/2}$. According to the decomposition in~\eqref{eqn:signal-noise-sum}, $W_t$ takes into account both the signal power $M_t M_t^T$ and also the noise power.
Therefore, the update~\eqref{eqn:sqrt-inv-conceptual} can be more effective than~\eqref{eqn:polar-conceptual} under nonuniform signal fractions.
Again, the update $W_t^{-1/2}M_t$ is invariant under positive scaling of $G_t$.

Now we can derive practical methods by replacing $M_t$ and $W_t$ in~\eqref{eqn:sqrt-inv-conceptual} with online estimators. The immediate choice is to use the two EMA estimators in~\eqref{eqn:ema-for-MW}.
Below are two more variations.
\begin{itemize}
\item
ASGO~\cite{an2025asgo} adds a small perturbation $\epsilon I$ to the second-moment estimator, i.e.,
\[
X_{t+1} = X_t - \alpha_t (\hW_t + \epsilon I)^{-1/2} \hM_t ,
\]
where $\epsilon>0$.
This is consistent with our assumption that $W_t$ is positive definite.
\item Leon~\cite{Jiangetal2026} employs a more sophisticated second-moment estimator:
\[
    X_{t+1} = X_t - \alpha_t (\gamma_t \hW_t + (1-\gamma_t) \hM_t \hM_t^T)^{-1/2} \hM_t,
\]
where $\gamma_t\in[0,1]$. It is an interpolation between Muon (as $\gamma_t\to 0$) and ASGO (as $\gamma_t\to 1$).
\end{itemize}

\paragraph{Two-sided preconditioning and the Shampoo optimizer}
As discussed in the beginning of Section~\ref{sec:preconditioning}, when $m>n$, it is natural to derive second-moment SA methods with right matrix multiplication as in~\eqref{eqn:right-mat-psa}.
In addition, we can also consider two-sided preconditioning:
\[
X_{t+1} = X_t - L_t G_t R_t ,
\]
where $L_t\in\R^{m\times m}$ and $R_t\in\R^{n\times n}$ are the left and right preconditioners. Following the derivations above, we can construct the following algorithm,
\[
X_{t+1} = X_t - \alpha_t W_t^{-1/4} M_t Q_t^{-1/4},
\]
where $W_t=\E_t[G_t G_t^T]$ and $Q_t=\E_t[G_t^T G_t]$ are the left and right second moments, appearing in~\eqref{eqn:2nd-moment-GGT} and~\eqref{eqn:2nd-moment-GTG} respectively.
Setting the two matrix powers to be $-1/4$ instead of $-1/2$ is necessary to retain the scaling-invariance property of the update.
The Shampoo optimizer \cite{gupta2018shampoo} is an instantiation of this conceptual method where $M_t$ is replaced by the simple estimator $G_t$ and $W_t$ and $Q_t$ are replaced by EMAs of $G_t G_t^T$ and $G_t^T G_t$ respectively.
Obviously, a simple variant can be obtained by replacing $M_t$ by an EMA of~$G_t$.

\subsection{Element-wise preconditioning}
\label{sec:elem-precond}
In this section, we focus on element-wise preconditioning based on the update rule
\[
X_{t+1} = X_t - P_t\odot G_t,
\]
where $P_t\in\R^{m\times n}$ has the same size as $X_t$ and $G_t$.
The derivation of the optimal preconditioner and the ensuing simplifications follow the same steps presented in Sections~\ref{sec:fused-precond}, \ref{sec:derive-muon} and~\ref{sec:derive-asgo}, only with simpler algebra.
Some results in this case were obtained in our previous work~\cite{jiang2025stochastic}, without relying on the concept of second-moment preconditioning.
Here we omit the details and only summarize the main results.

First, following similar steps as in Section~\ref{sec:fused-precond}, we obtain a conceptual second-moment method
\begin{equation}\label{eqn:elem-conceptual}
X_{t+1} = X_t - \alpha_t |M_t| \oslash W_t \odot M_t,
\end{equation}
where $M_t=\E_t[G_t]$ and $W_t=\E_t[G_t\odot G_t]$.
Since $|M_t|=\sqrt{M_t\odot M_t}$, this method has clear correspondence with~\eqref{eqn:general-conceptual-method}.
The \emph{signal fraction} matrix in this case is also defined element-wise:
\[
S_t:= \bigl(M_t \odot M_t\bigr) \oslash W_t.
\]
All elements of the SiF matrix $S_t$ take values between $0$ and $1$, with $0$ indicating $G_t$ being pure noise and $1$ indicating pure signal.
Since $|M_t|\odot M_t=(M_t\odot M_t)\odot \sign(M_t)$, we can write~\eqref{eqn:elem-conceptual} as
\begin{equation}\label{eqn:elem-full-sif}
X_{t+1} = X_t - \alpha_t S_t\odot \sign(M_t) .
\end{equation}
The following two simplifications parallel those in Sections~\ref{sec:derive-muon} and~\ref{sec:derive-asgo} respectively.

\paragraph{Sign gradient method and variants}
When the signal fractions are relatively uniform, i.e., $S_t\approx\gamma_t\ones_{m\times n}$ for some $\gamma_t>0$, we can assimilate $\gamma_t$ into the tunable step size $\alpha_t$ and arrive at
\begin{equation}\label{eqn:elem-sign}
X_{t+1} = X_t - \alpha_t\, \sign (M_t).
\end{equation}
With different estimators for the first moment $M_t$, we recover several popular practical methods:
\begin{itemize}
\item SignSGD~\cite{fabian1960stochastic} simply uses $G_t$ as an estimator of~$M_t$.
\item Signum~\cite{bernstein2018signsgd} employs an EMA of $G_t$ as an estimator of $M_t$.
\item Lion~\cite{chen2023symbolic} can be interpreted as using a combination of $G_t$ and its EMA as an estimator for $M_t$, essentially replacing $\ortho(\cdot)$ in~\eqref{eqn:qhm} with the $\sign(\cdot)$ function.
\end{itemize}

\paragraph{Adam and variants}
When the signal fractions are not uniform, we should take account of their effect but also need to compensate for the inaccuracy and volatility brought by any practical estimators. A reasonable compromise is to only factor in their square root (mapping them closer to uniform).
More concretely, we would like to keep the square root of $S_t$ in~\eqref{eqn:elem-full-sif} instead of replacing it with a uniform constant as done in~\eqref{eqn:elem-sign}. This leads to
\begin{align*}
X_{t+1} &= X_t - \alpha_t \sqrt{S_t}\odot \sign(M_t) \\
&= X_t - \alpha_t M_t \oslash \sqrt{W_t}.
\end{align*}
Replacing $M_t$ and~$W_t$ with respective estimators, we recover several practical methods:
\begin{itemize}
\item
Adam~\cite{KingmaBa2014adam} uses EMA estimators for both~$M_t$ and~$W_t$ but with an $\epsilon>0$ added to the denominator, as shown in~\eqref{eqn:adam}. We can also view $\hW_t+\epsilon$ as an estimator for $W_t$.
\item
RMSProp~\cite{rmsprop2012lecture} uses $G_t$ to estimate $M_t$ but the same second-moment estimator as Adam.
AdaGrad~\cite{DuchiHazanSinger2011adagrad} can be interpreted as using a different second-moment estimator.
\item
Nadam~\cite{dozat2016incorporating} uses a combination of $G_t$ and its EMA as the estimator for $M_t$, similar to the one in~\eqref{eqn:qhm}.
AdEMAMix~\cite{pagliardini2024ademamix} combines two EMAs of $G_t$ with different smoothing factors.
\end{itemize}
The above examples mainly differ in their choice of first-moment estimators and share the same second-moment estimator as Adam.
Exploring more varieties of the second-moment estimator could also be fruitful.
Beyond the generic EMA estimators, our framework also facilitates the adoption of more customized estimators that can be model-specific and data-driven.

\section{A general framework for convergence analysis}
\label{sec:two-stage-framework}

In this section, we present a general framework for the convergence analysis of second-moment SA methods.
For full generality, our framework also covers \emph{decoupled weight decay}~\cite{Loshchilov2019AdamW}, which has proven to be a very effective ingredient in deep learning optimizers, especially for improving the performance of both Adam~\cite{Loshchilov2019AdamW} and Muon~\cite{liu2025muon}.
Specifically, decoupled weight decay means including an additional damping term $\lambda X_t$, where $\lambda\geq 0$, to the update in~\eqref{eqn:practical-smsa}:
\begin{align}
X_{t+1} &= X_t - \alpha_t\bigl(\phi(\hM_t,\hW_t) + \lambda X_t\bigr) \nonumber\\
&= (1-\alpha_t\lambda)X_t - \alpha_t \phi(\hM_t,\hW_t).
\label{eqn:practical-smsa-dwd-2line}
\end{align}
The damping term can be interpreted through the classical Halpern iteration~\cite{halpern1976fixed} or, in the context of machine learning, understood as the consequence of quadratic regularization.

As explained in Section~\ref{sec:two-stage-intro}, our framework consists of two stages.
First, we analyze the conceptual counterpart of~\eqref{eqn:practical-smsa-dwd-2line}:
\begin{equation}\label{eqn:conceptual-smsa-dwd}
X_{t+1} = (1-\alpha_t \lambda) X_t - \alpha_t \phi\bigl(M_t, W_t\bigr).
\end{equation}
(Recall that $M_t$ and $W_t$ are simple notations for $M(X_t)$ and $W(X_t)$ respectively.)
The analysis relies on an aiming condition for the conceptual update $\phi(M_t,W_t)$, which we introduce in Section~\ref{sec:aiming_general-smsa}. Then we prove the convergence of~\eqref{eqn:conceptual-smsa-dwd} and characterize its rate of convergence in Section~\ref{sec:analysis-conceptual}.
In the second stage, presented in Section~\ref{sec:analysis-practical}, we analyze the practical algorithms~\eqref{eqn:practical-smsa-dwd-2line} by assuming that the deviation between the expected practical update $\E_t[\phi(\hM_t,\hW_t)]$ and the conceptual update $\phi(M_t,W_t)$ is bounded almost surely.

\subsection{Aiming condition}\label{sec:aiming_general-smsa}
For the classical Robbins-Monro method~\eqref{eqn:sa}, the following is a typical set of assumptions that guarantee almost sure iterate convergence:
\begin{itemize}
    \item The random variables $\{\xi_t\}_{t\geq 0}$ are independent and $G(X, \xi)$ has bounded variance;
    \item The step sizes $\{\alpha_t\}_{t \geq 0}$ are nonnegative and satisfy $\sum_{t=0}^{\infty} \alpha_t = \infty$ and $\sum_{t=0}^{\infty} \alpha_t^2 < \infty$;
    \item Suppose $M(X_\ast):=\E_{\xi_\ast}[G(X_\ast, \xi_\ast)] = 0$ and there exists $L, \mu > 0$ such that for all $X \in \R^{m \times n}$:
    \begin{align}
        \langle M(X), \,X - X_\ast \rangle &\geq \mu \norm{X-X_\ast}_F^2, \label{eqn:aiming-sakrison} \\
        \norm{M(X)}_F &\leq L \norm{X- X_\ast}_F. \label{eqn:growth-sakrison}
    \end{align}
\end{itemize}
Under these conditions, the sequence $\{X_t\}$ generated by~\eqref{eqn:sa} converges to $X_\ast$ almost surely~\cite{Sakrison1966, ROBBINS1971}.
Specifically, condition~\eqref{eqn:aiming-sakrison} ensures that the update direction $M(X)=\E_\xi [G(X, \xi)]$ always aims towards the target $X_\ast$. For this reason, we call it the \emph{aiming condition}.
Condition~\eqref{eqn:growth-sakrison} controls the growth of $\|M(X)\|_F$ as~$X$ deviates from~$X_*$. Together, they guarantee that moving along the update direction $-M(X)$ reduces the distance to $X_*$ for sufficiently small $\alpha_t$.

\begin{figure}[t]
    \centering
    \begin{tikzpicture}[>=stealth]

\draw[->] (-3,0) -- (4,0) node[right] {$X$};
\draw[->] (0,-2.5) -- (0,4) node[above] {$\nabla F$};

\draw[thick, blue!70!black] plot[smooth, tension=0.8] coordinates {
    (-2.5, -1.) (-1.5, 0) (0.0, 1.5) (1.5, 1.)  
} node[right] { };

\draw[thick, blue!70!black] plot[smooth, tension=0.] coordinates {
    (1.5, 1.) (1.5, 2.5)
} node[right] { };

\draw[thick, blue!70!black] plot[smooth, tension=0.9] coordinates {
    (1.5, 2.5) (2., 1.5) (3, 3.) 
} node[right] {$\nabla F(X)$};

\filldraw (-1.5,0) circle (1.5pt);
\node[below] at (-1.2,0) {$X_\ast$};

\draw[dashed, red!80!black] (-2.5,-2) -- (0.5,4.0) node[pos=0.8, left, align=right] {$L(X-X_\ast)$};

\draw[dashed, green!60!black] (-2.5,-0.2) -- (4,1.1) node[pos=0.9, below right] {$\mu(X-X_\ast)$};

\end{tikzpicture}
    \caption{One-dimensional depiction of the gradient of a function $F$ that satisfies both the aiming condition~\eqref{eqn:aiming-opt} and the growth condition~\eqref{eqn:growth-opt} but is neither monotone nor Lipschitz continuous.}
    \label{fig:aiming_sakrison}
\end{figure}
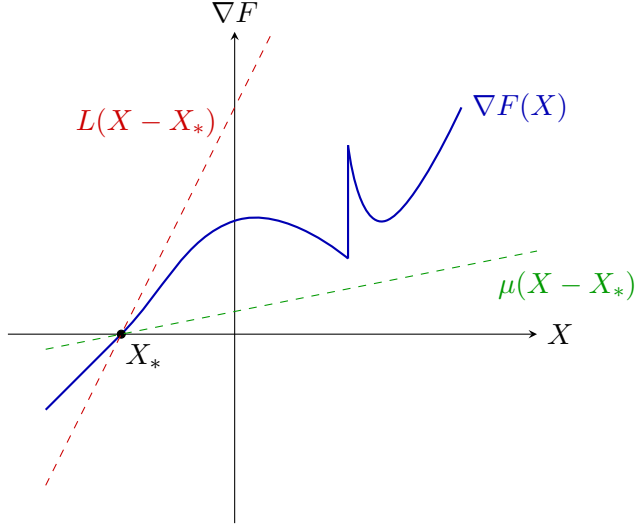

In the context of optimization, i.e., if $M$ is the gradient of a cost function $F:\R^{m\times n}\to\R$, the two conditions~\eqref{eqn:aiming-sakrison} and~\eqref{eqn:growth-sakrison} become
\begin{align}
    \langle \nabla F(X), \, X - X_\ast \rangle &\geq \mu \norm{X-X_\ast}_F^2,
\label{eqn:aiming-opt}\\
    \norm{\nabla F(X)}_F &\leq L \norm{X- X_\ast}_F,
\label{eqn:growth-opt}
\end{align}
which are consequences of $F$ being $\mu$-strongly convex and $L$-smooth respectively \cite[see, e.g.,][]{Nesterov04book}.
Notice that these two conditions
are much weaker than strong convexity and smoothness because they are anchored at a specific point $X_*$.
Figure~\ref{fig:aiming_sakrison} depicts a one-dimensional gradient mapping $\nabla F(X)$ that is sandwiched between two linear functions $L(X-X_*)$ and $\mu(X-X_*)$, satisfying both the aiming and growth conditions. However, $\nabla F(X)$ is neither monotone (implying that $F$ is not convex) nor Lipschitz-continuous (thus $F$ is nonsmooth).

To analyze second-moment SA methods, we extend the aiming condition~\eqref{eqn:aiming-sakrison} and the growth condition~\eqref{eqn:growth-sakrison} to work with the blending function $\phi$ as well as decoupled weight decay.
\begin{assumption}[Aiming and growth conditions]\label{assum:aiming-smsa-dwd}
There exist $X_*\in\R^{m\times n}$, $\mu>0$, and $L,C\geq0$ such that, for all $X\in\R^{m\times n}$,
\begin{align}
    \left\langle \phi\bigl(M(X), W(X)\bigr) + \lambda X, \,  X-X_*\right\rangle &\geq \mu \|X - X_*\|_F^2, \label{eqn:aiming-smsa-dwd}\\
    \norm{\phi\bigl(M(X), W(X)\bigr)}_F &\leq L \norm{X-X_\ast}_F + C.
    \label{eqn:growth-smsa-dwd}
\end{align}
\end{assumption}

Unlike the first-moment aiming condition~\eqref{eqn:aiming-sakrison}, the generalization in~\eqref{eqn:aiming-smsa-dwd} does not have a simple connection with convexity (when interpreting $M$ as the gradient mapping of a cost function) due to the presence of~$\phi$ and~$W$.
It is easy to construct examples that satisfy one but fail the other \cite[][Appendix~A]{jiang2025stochastic}.
Nonetheless, it is closely related to the definition of \emph{strong pseudogradient} \cite[][Section~2.2.3]{Polyak1987}. A matrix mapping $D(X)$ is a strong pseudogradient relative to a Lyapunov function $V(X)$ if there exists $\gamma>0$ such that
\[
\left\langle D(X),\nabla V(X)\right\rangle \geq\gamma V(X).
\]
Condition~\eqref{eqn:aiming-smsa-dwd} implies that $\phi(M(X),W(X))+\lambda X$ is a strong pseudogradient with respect to $V(X)=(1/2)\|X-X_*\|_F^2$ with parameter $\gamma=2\mu$.

The generalized growth condition~\eqref{eqn:growth-smsa-dwd} offers more flexibility than its classical counterpart~\eqref{eqn:growth-sakrison}.
When $C > 0$, we can set $L=0$ if $\phi$ is a bounded mapping, which is the case for $\sign(M(X))$ and $\ortho(M(X))$.
More generally, the composite mapping $\phi(M(X), W(X))$ may remain bounded even if~$\phi$ itself is not. As an example, consider the update mapping in~\eqref{eqn:sqrt-inv-conceptual},
\[
\phi(M(X),W(X)) = W(X)^{-1/2} M(X).
\]
While $\phi(Y,Z)=Y^{-1/2}Z$ is not bounded in general, the fact that $M(X)$ and $W(X)$ are the first and second moments of the same random variable guarantees that $\phi(M(X),W(X))$ is always bounded.
The same conclusion is true for Adam's update map
\[
\phi\bigl(M(X), W(X)\bigr) = M(X) \oslash \sqrt{W(X)}.
\]

Indeed, having a bounded update map is a common property for the general class of second-moment SA method we consider in this paper. It is rooted from the derivation of the optimal preconditioner in Section~\ref{sec:fused-precond}, where the optimal update $P_t^\star G_t$ is scaling invariant. This property persisted in all our approximation and simplifications in deriving Muon, ASGO, Adam and their variants.
As a consequence, the growth condition~\eqref{eqn:growth-smsa-dwd} holds automatically for all of them, even if the first-moment $M(X)$ does not satisfy~\eqref{eqn:growth-sakrison}.
\emph{This is a major advantage of second-moment SA methods over the classical first-moment methods.}

\begin{remark}\label{rem:need-regularization}
When $\phi(M(X),W(X))$ is bounded by a constant, we need $\lambda>0$ in order for the aiming condition~\eqref{eqn:aiming-smsa-dwd} to hold globally. This provides an explanation for the use of decoupled weight decay \cite{Loshchilov2019AdamW} in second-moment methods such as Adam and Muon.
While it is difficult to characterize the exact range of~$\lambda$ that guarantees the global aiming condition, in practice we only need this condition to hold locally within a sufficiently large neighborhood of $X_*$.
\end{remark}

\begin{example}[Aiming condition of a matrix polynomial for Muon]
\label{ex:matrix-polynomial}
Assume $m\leq n$ and fix some $X_*\in\R^{m\times n}$.
Consider the matrix polynomial
\[
M(X)=\bigl((X-X_*)(X-X_*)^T\bigr)^p(X-X_*),
\]
where $p\geq-1/2$.
Given the SVD of $X-X_*=U\Sigma V^T$ where $U\in \R^{m \times m}$ and $V \in \R^{n \times m}$, we have $M(X) = U\bigl(\Sigma^{2p+1}\bigr)V^T$ and thus $\polar(M(X))=UV^T$. This leads to
\[
\bigl\langle\polar(M(X)),X-X_*\bigr\rangle
=\trace(\Sigma) = \|X-X_*\|_{\text{tr}}.
\]
where $\|X-X_*\|_\text{tr}$ denotes the trace norm (nuclear norm) of $X-X_*$.
We deduce that the aiming condition~\eqref{eqn:aiming-smsa-dwd} with $\lambda=0$ and any $\mu >0$ holds in the following neighborhood of $X_*$:
\[
\left\{X\in\R^{m\times n} : \|X-X_*\|_F^2\leq\frac{1}{\mu} \|X-X_*\|_\text{tr} \right\}.
\]
The global aiming condition holds with sufficiently large $\lambda>0$.
Meanwhile, the growth condition~\eqref{eqn:growth-smsa-dwd} holds with $L=0$ and $C=\sqrt{m}$, because $\|\ortho(M(X))\|_F^2=\|UV^T\|_F^2\leq m$.
\end{example}

\subsection{Analysis of conceptual algorithms}
\label{sec:analysis-conceptual}
In this section, we prove convergence and establish the rates of convergence for the conceptual method~\eqref{eqn:conceptual-smsa-dwd} under Assumption~\ref{assum:aiming-smsa-dwd}. Our first result is the following one-step contraction property.

\begin{lemma}\label{lem:conceptual-smsa-dwd-one-step}
Suppose Assumption~\ref{assum:aiming-smsa-dwd} holds, $\alpha_t\geq0$, and $\alpha_t\lambda\leq1$ for all $t\geq0$. Then the sequence $\{X_t\}$ generated by~\eqref{eqn:conceptual-smsa-dwd} satisfies the following one-step improvement bound for all $t\geq 0$,
\begin{equation}\label{eqn:conceptual-smsa-dwd_contraction}
\norm{X_{t+1}-X_\ast}_F^2
\leq \bigl(1 + \alpha_t^2 (3L^2 + 2 \lambda \mu) - 2\alpha_t \mu \bigr)\norm{X_{t}-X_\ast}_F^2 + 3\alpha_t^2 (\lambda^2 \norm{X_*}_F^2 + C^2 ).
\end{equation}
\end{lemma}

\begin{proof}
First, we notice that the aiming condition~\eqref{eqn:aiming-smsa-dwd} in Assumption~\ref{assum:aiming-smsa-dwd} is equivalent to
\begin{equation}\label{eqn:aiming-smsa-dwd-x*}
\left\langle X_t-X_*, \, \phi\bigl(M_t, W_t\bigr) + \lambda X_* \right\rangle \geq (\mu - \lambda) \|X_t - X_*\|_F^2.
\end{equation}
To see this, we simply add and subtract $\lambda X_t$ inside the inner product:
\begin{align*}
\left\langle X_t-X_*, \, \phi\bigl(M_t, W_t\bigr) + \lambda X_* \right\rangle
&=
\left\langle X_t-X_*, \, \phi\bigl(M_t, W_t\bigr) + \lambda X_t - \lambda X_t + \lambda X_* \right\rangle \\
&=\left\langle X_t-X_*, \, \phi\bigl(M_t, W_t\bigr) + \lambda X_t \right\rangle - \lambda \|X_t - X_*\|^2_F,
\end{align*}
which together with~\eqref{eqn:aiming-smsa-dwd} implies~\eqref{eqn:aiming-smsa-dwd-x*}. Given the conceptual update rule~\eqref{eqn:conceptual-smsa-dwd}, we have
\begin{align}
\norm{X_{t+1} - X_\ast}^2_F
    &=\norm{(1 - \alpha_t \lambda) X_{t} - \alpha_t \phi\bigl(M_t, W_t\bigr) - X_*}^2_F\nonumber\\
    &=\norm{(1 - \alpha_t \lambda) (X_{t} - X_*) - \alpha_t \bigl( \phi\bigl(M_t, W_t \bigr) + \lambda X_*\bigr) }^2_F\nonumber\\
    &= (1 - \alpha_t \lambda)^2 \norm{X_{t} - X_*}_F^2
    - 2 \alpha_t (1 - \alpha_t \lambda) \left \langle X_{t} - X_*, \phi\bigl(M_t, W_t\bigr) + \lambda X_* \right \rangle \nonumber\\
    &\qquad + \alpha_t^2\norm{\phi\bigl(M_t, W_t \bigr) + \lambda X_* }^2_F. \label{eqn:conceptual-smsa-dwd-next-dist}
\end{align}
For the second term on the right-hand side, we use the assumption $\alpha_t\lambda\leq 1$ and~\eqref{eqn:aiming-smsa-dwd-x*} to obtain
\begin{align*}
-2\alpha_t(1-\alpha_t\lambda)\left \langle X_{t} - X_*, \phi\bigl(M_t, W_t\bigr) + \lambda X_* \right \rangle \leq  -2\alpha_t(1-\alpha_t\lambda)(\mu - \lambda) \|X_t - X_*\|_F^2.
\end{align*}
For the third term, applying the triangle inequality and the growth condition~\eqref{eqn:growth-smsa-dwd} yields
\begin{align*}
    \norm{\phi\bigl(M_t, W_t\bigr) + \lambda X_*}^2_F &\leq  \bigl(\norm{\phi\bigl(M_t, W_t\bigr)}_F + \lambda \norm{X_*}_F\bigr)^2 \\
    &\leq (L \norm{X_t - X_\ast} + C + \lambda \norm{X_*}_F)^2\\
    &\leq 3 L^2 \norm{X_t - X_\ast}_F^2 + 3C^2 + 3\lambda^2 \norm{X_*}_F^2,
\end{align*}
where the last inequality is due to the fact $(a+b+c)^2 \leq 3a^2 + 3b^2 +3c^2$. Substituting these bounds into~\eqref{eqn:conceptual-smsa-dwd-next-dist}, we get:
\begin{align*}
\norm{X_{t+1} - X_\ast}^2_F
    &\leq (1 - \alpha_t \lambda)^2 \norm{X_{t} - X_*}_F^2 - 2 \alpha_t (1 - \alpha_t \lambda) (\mu - \lambda) \|X_t - X_*\|_F^2 \\
    &\qquad + 3 \alpha_t^2 L^2 \norm{X_t - X_\ast}_F^2 + 3\alpha_t^2  C^2 + 3\alpha_t^2 \lambda^2 \norm{X_*}_F^2\\
    & = \bigl(1 - 2 \alpha_t \mu + 2 \alpha_t^2 \lambda \mu - \alpha_t^2 \lambda^2 + 3 \alpha_t^2 L^2 \bigr)\norm{X_{t}-X_\ast}_F^2 + 3\alpha_t^2 (\lambda^2 \norm{X_*}_F^2 + C^2 ).
\end{align*}
Dropping the negative coefficient $- \alpha_t^2 \lambda^2$ for $\norm{X_{t}-X_\ast}_F^2$ yields the desired result.
\end{proof}

We now establish convergence for the conceptual algorithm under Asspumption~\ref{assum:aiming-smsa-dwd} and appropriate assumptions on the step sizes.

\begin{theorem}[Convergence of conceptual method] \label{thm:conceptual-smsa-dwd_convergence}
Suppose the step size sequence $\{\alpha_t\}_{t \geq 0}$ satisfies
\begin{equation}\label{eqn:lr_schedule_conceptual_practical-smsa-dwd}
\alpha_t > 0, \qquad \alpha_t \to 0, \qquad \alpha_t \lambda \leq 1,
\qquad \sum_{t=0}^{\infty} \alpha_t = \infty.
\end{equation}
Then under Assumption~\ref{assum:aiming-smsa-dwd}, the sequence $\{X_t\}$ generated by~\eqref{eqn:conceptual-smsa-dwd} satisfies $\norm{X_t - X_\ast}_F \to 0$.
\end{theorem}

The proof for Theorem~\ref{thm:conceptual-smsa-dwd_convergence} is based on the following lemma in~\cite{Polyak1987}.

\begin{lemma}[{\cite[][Section~2.2, Lemma~6']{Polyak1987}}]
\label{lem:polyak_numeric_lem6}
    Let $\{d_t\}_{t \geq 0}$ be a nonnegative sequence satisfying
    \begin{equation}\label{eqn:polyak_recursion_lem6}
        d_{t+1} \leq (1+a_t) d_t - b_t h(d_t) + c_t,
    \end{equation}
    where $a_t, b_t, c_t \geq 0$ satisfy
    \[
    \sum_{t=0}^{\infty} b_t = \infty, \qquad a_t \to 0, \qquad c_t \to 0, \qquad \frac{a_t}{b_t} \to 0, \qquad \frac{c_t}{b_t} \to 0,
    \]
and~$h$ is a convex function satisfying
    \[
    h(d) > 0 \quad \forall d > 0, \qquad h(0)=0, \qquad h(d') \geq h(d) \quad \forall d' \geq d \geq 0.
    \]
    Then $d_t \to 0$.
\end{lemma}

\begin{proof}[Proof of Theorem~\ref{thm:conceptual-smsa-dwd_convergence}.] Define $d_t= \norm{X_t - X_\ast}_F^2$. Lemma~\ref{lem:conceptual-smsa-dwd-one-step} implies the following recursion:
\begin{align*}
d_{t+1} &= \bigl(1+\alpha_t^2 (3L^2 + 2 \lambda \mu)\bigr) d_t - 2\alpha_t\mu \, d_t +3\alpha_t^2 (\lambda^2 \norm{X_*}_F^2 + C^2) .
\end{align*}
which is in the form of~\eqref{eqn:polyak_recursion_lem6} with $h$ being the identity map and
\[
a_t =\alpha_t^2 (3L^2 + 2 \lambda \mu), \qquad b_t = 2\alpha_t \mu, \qquad c_t = 3\alpha_t^2 (\lambda^2 \norm{X_*}_F^2 + C^2).
\]
Notice that the sequences $a_t$, $b_t$, and~$c_t$ are all non-negative and the assumption on $\alpha_t$ in~\eqref{eqn:lr_schedule_conceptual_practical-smsa-dwd} guarantees that they satisfy all the requirements in Lemma~\ref{lem:polyak_numeric_lem6}.
Moreover, the identity map is convex, monotonically increasing and attains~$0$ at~$0$. Therefore, we may apply the lemma to conclude that $d_t = \norm{X_t - X_\ast}_F^2 \to 0$.
\end{proof}

\begin{remark}
Lemma~\ref{lem:polyak_numeric_lem6} is a deterministic counterpart of the Robbins-Siegmund lemma \cite{ROBBINS1971}, which is a powerful tool in establishing almost sure convergence of stochastic approximation methods.
We used the Robbins-Siegmund lemma in our previous work \cite{jiang2025stochastic} where the conceptual method is also a stochastic algorithm. In this paper, we work with deterministic conceptual algorithms (which have more flexibility in modeling different algorithms) and thus resort to Lemma~\ref{lem:polyak_numeric_lem6}.
\end{remark}

\smallskip
Our next two results concern the rate of convergence of the conceptual method. First, for stepsizes $\alpha_t$ following a $1/t$ schedule, we establish an asymptotic $\cO(1/t)$ rate of convergence in squared distance $\norm{X_t - X_*}^2_F$, as given in the following theorem.

\begin{theorem}\label{thm:conceptual-smsa-dwd_convergence_1/t}
Consider the algorithm~\eqref{eqn:conceptual-smsa-dwd} with the step size rule $\alpha_t = \frac{\alpha}{t+1}$. Suppose Assumption~\ref{assum:aiming-smsa-dwd} holds and $\alpha$ satisfies $\frac{1}{2\mu} < \alpha \leq \frac{2\mu}{3L^2 + 2 \lambda \mu}$ where we assume $3L^2+2\lambda\mu<4\mu^2$.
Then as $t \to \infty$,
    \[
    \norm{X_t - X_*}^2_F \leq \frac{\alpha^2 A}{2 \mu \alpha - 1} \frac{1}{t+1} + \mathcal O \left(\frac{1}{(t+1)^{2}} + \frac{1}{(t+1)^{2 \mu \alpha}}\right),
    \]
    where
    \[A= (3L^2 + 2 \lambda \mu)B + 3 \lambda^2 \norm{X_*}_F^2 + 3 C^2, \qquad
    B=\norm{X_{0}-X_\ast}_F^2 + \frac{1}{2}\bigl(\lambda^2 \norm{X_*}_F^2 + C^2 \bigr) \alpha^2 \pi^2.\]
\end{theorem}

As discussed in Section~\ref{sec:aiming_general-smsa}, we have $L=0$ for a broad class of second-moment SA methods. In this case, the assumption $3L^2+2\lambda\mu<4\mu^2$ boils down to $\lambda<2\mu$, and the upper bound on $\alpha$ becomes $\alpha\leq 1/\lambda$ which translates into $\alpha_t\lambda\leq 1$ for all $t\geq 0$.

The proof of Theorem~\ref{thm:conceptual-smsa-dwd_convergence_1/t} is based on Lemma~\ref{lem:conceptual-smsa-dwd-one-step} and a classical result of Chung \cite{chung1954}.

\begin{lemma}[{\cite[Lemma~1]{chung1954}}]
\label{lem:chung1}
    Suppose $\{d_t\}$ is a sequence of real numbers such that for all $t\geq 1$,
    \begin{equation}\label{eqn:chung1-recursion}
        d_{t+1} \leq \left(1 - \frac{a}{t}\right) d_t + \frac{b}{t^{p+1}},
    \end{equation}
where $a > p > 0$ and $b > 0$. Then
    \[
    d_t \leq \frac{b}{a-p} \frac{1}{t^p} + \mathcal O \left(\frac{1}{t^{p+1}} + \frac{1}{t^a}\right).
    \]
\end{lemma}

In order to apply Chung's lemma, we need to first derive a uniform upper bound on $\norm{X_{t}-X_\ast}_F^2$ for any $t\geq 0$, which will be incorporated into the constant~$b$ in the lemma. To this end, we start with the one-step improvement inequality~\eqref{eqn:conceptual-smsa-dwd_contraction}:
\[
\norm{X_{t+1}-X_\ast}_F^2
\leq \bigl(1 + \alpha_t\bigl( \alpha_t (3L^2 + 2 \lambda \mu) - 2\mu\bigr)\bigr)\norm{X_{t}-X_\ast}_F^2 + 3\alpha_t^2 (\lambda^2 \norm{X_*}_F^2 + C^2 ).
\]
If $\alpha_t \leq \frac{2\mu}{3L^2 + 2 \lambda \mu}$, we can bound the coefficient of $\norm{X_{t}-X_\ast}_F^2$ by~$1$. Therefore,
\begin{align}
\norm{X_{t}-X_\ast}_F^2
& \leq \norm{X_{t-1}-X_\ast}_F^2 + 3\alpha_{t-1}^2 (\lambda^2 \norm{X_*}_F^2 + C^2 ) \nonumber\\
&\leq \norm{X_{0}-X_\ast}_F^2 + 3 (\lambda^2 \norm{X_*}_F^2 + C^2) \sum_{s=0}^{t-1} \alpha_s^2,
\label{eqn:dist-upper-bound}
\end{align}
where the last term is bounded if $\sum_{s=0}^{\infty} \alpha_s^2 < \infty$, which is true  for $\alpha_s=\frac{\alpha}{s+1}$.

\begin{proof}[Proof of Theorem~\ref{thm:conceptual-smsa-dwd_convergence_1/t}]
Under the step size rule $\alpha_t = \frac{\alpha}{t+1}$ with $\alpha \leq \frac{2\mu}{3L^2 + l \lambda \mu}$, we have $\alpha_t \leq \frac{2\mu}{3L^2 + 2 \lambda \mu}$ for all $t\geq 0$ and $\sum_{s=0}^{\infty} \alpha_s^2 \leq \frac{\alpha^2 \pi^2}{6}$. Therefore, the inequality~\eqref{eqn:dist-upper-bound} implies the following bound:
\begin{equation}\label{eqn:dist-upper-bound-1/t}
    \norm{X_{t}-X_\ast}_F^2 \leq \norm{X_{0}-X_\ast}_F^2 + \frac{1}{2} \bigl(\lambda^2 \norm{X_*}_F^2 + C^2 \bigr) \alpha^2 \pi^2=B.
\end{equation}
Substituting~\eqref{eqn:dist-upper-bound-1/t} and $\alpha_t = \frac{\alpha}{t+1}$ into~\eqref{eqn:conceptual-smsa-dwd_contraction} and noticing the definition of~$A$, we obtain
\begin{align}
\norm{X_{t+1}-X_\ast}_F^2
& \leq  \left(1 - 2 \mu \alpha_t \right)\norm{X_{t}-X_\ast}_F^2 + \alpha_t^2 \left((3L^2 + 2 \lambda \mu)B + 3 \lambda^2 \norm{X_*}_F^2 + 3C^2\right) \nonumber\\
& = \left(1 - \frac{2 \mu \alpha}{t+1}\right)\norm{X_{t}-X_\ast}_F^2 + \frac{\alpha^2}{(t+1)^2} A. \nonumber
\end{align}
Let $d_{t+1}=\norm{X_{t}-X_\ast}_F^2$, $a=2 \mu \alpha$, $p=1$ and $b=\alpha^2 A$.
The above inequality can be written in the form of~\eqref{eqn:chung1-recursion} with shifted index $t\to t+1$.
Applying Lemma~\ref{lem:chung1} gives the desired result.
\end{proof}

\smallskip

Theorem~\ref{thm:conceptual-smsa-dwd_convergence_1/t} requires the lower bound $\alpha>1/(2\mu)$ to obtain the $\cO(1/t)$ rate of convergence. However, it can be hard to obtain a precise estimate of~$\mu$ in order to satisfy this condition. We can remove this requirement by working with the step size rule $\alpha_t=\frac{\alpha}{(t+1)^p}$ with some $p\in(\frac{1}{2},1)$. In this case we have an \emph{asymptotic} $\cO(1/t^p)$ rate, as stated in the following theorem.

\begin{theorem}\label{thm:conceptual-smsa-dwd_1/t^p}
Consider the algorithm~\eqref{eqn:conceptual-smsa-dwd} with step size rule $\alpha_t = \frac{\alpha}{(t+1)^p}$ where $p \in \left(\frac{1}{2},1\right)$ and the constant $\alpha$ satisfies $0 < \alpha \leq \frac{2\mu}{3L^2 + 2 \lambda \mu}$.
Then under Assumption~\ref{assum:aiming-smsa-dwd}, we have
\begin{equation}\label{eqn:conceptual-smsa_1/t^p}
\limsup_{t\to\infty}\, (t+1)^p \norm{X_t-X_\ast}_F^2
\leq\frac{\alpha A_p}{2\mu},
\end{equation}
where
\[A_p= (3L^2 + 2 \lambda \mu)B_p + 3 \lambda^2 \norm{X_*}_F^2 + 3C^2, \qquad
B_p=\norm{X_{0}-X_\ast}_F^2 +  \alpha^2 \bigl(\lambda^2 \norm{X_*}_F^2 + C^2\bigr) \sum_{s=1}^{\infty} \frac{1}{s^{2p}}. \]
\end{theorem}
The proof is based on another lemma of Chung in the same paper\cite{chung1954}, restated below.

\begin{lemma}[{\cite[Lemma 4]{chung1954}}]
\label{lem:chung_lem4}
    Suppose $\{d_t\}$ is a sequence of real numbers such that for all $t\geq 1$,
    \begin{equation}\label{eqn:chung_inequality_cpt3_lem4}
        d_{t+1} \leq \left(1 - \frac{a_t}{t^p}\right) d_t + \frac{b}{t^{q}},
    \end{equation}
    where $0< p < 1, p<q, a_t \geq a > 0, b>0$. Then
    \[
    \limsup_{t \to \infty}\, t^{q-p} d_t \leq \frac{b}{a}.
    \]
\end{lemma}

\begin{proof}[Proof of Theorem~\ref{thm:conceptual-smsa-dwd_1/t^p}]
Under the step size rule $\alpha_t= \frac{\alpha}{(t+1)^p}$ with $\alpha \leq \frac{2\mu}{3L^2 + 2 \lambda \mu}$ and $p \in (\frac{1}{2}, 1)$, we have $\alpha_t \leq\frac{2\mu}{3L^2 + 2 \lambda \mu}$ for all $t\geq 0$ and $\sum_{s=0}^{\infty} \alpha_s^2 = \alpha^2\sum_{s=1}^{\infty} \frac{1}{s^{2p}} < \infty$. Therefore,~\eqref{eqn:dist-upper-bound} holds and implies the following upper bound for $\norm{X_t-X_*}_F^2$ for all $t\geq 0$:
\[
\norm{X_{t}-X_\ast}_F^2 \leq \norm{X_{0}-X_\ast}_F^2 +  3 \alpha^2 (\lambda^2 \norm{X_*}_F^2 + C^2) \sum_{s=0}^{\infty} \frac{1}{(s+1)^{2p}}=B_p.
\]
Substituting the inequality above into~\eqref{eqn:conceptual-smsa-dwd_contraction}
and using $\alpha_t=\frac{\alpha}{(t+1)^p}$ gives
\begin{align*}
\norm{X_{t+1}-X_\ast}_F^2
& \leq  \left(1 - 2 \mu \alpha_t \right)\norm{X_{t}-X_\ast}_F^2 + \alpha_t^2 \left((3L^2 + 2 \lambda \mu)B_p + 3 \lambda^2 \norm{X_*}_F^2 + 3C^2\right) \\
& =  \left(1 - \frac{2 \mu \alpha}{(t+1)^p}\right)\norm{X_{t}-X_\ast}_F^2 + \frac{\alpha^2}{(t+1)^{2p}}A_p.
\end{align*}
Let $d_{t+1}=\norm{X_{t}-X_\ast}_F^2$, $a=2 \mu \alpha$, $b=\alpha^2 A_p$ and $q=2p$. Then the last display becomes~\eqref{eqn:chung_inequality_cpt3_lem4} with shifted index $t\to t+1$ and applying Lemma~\ref{lem:chung_lem4} yields the desired result.
\end{proof}

\subsection{Convergence of practical algorithms}
\label{sec:analysis-practical}

Moving on to the second stage of our analysis, we focus on algorithms of the form
\begin{equation}\label{eqn:practical-smsa-dwd}
X_{t+1} = (1-\alpha_t\lambda)X_t - \alpha_t \phi(\hM_t,\hW_t).
\end{equation}
We assume that the estimators $\hM_t$ and $\hW_t$ are constructed from $X_0,\ldots,X_t$ and $G_0,\ldots,G_t$. Consequently, $\hM_t$, $\hW_t$ and $\phi(\hM_t,\hW_t)$ are measurable with respect to the filtration $\mathcal{F}_{t+1}$ where
\begin{equation}\label{eqn:filtration}
\mathcal F_t:=\sigma(X_0,\xi_0,X_1,\ldots,\xi_{t-1},X_t).
\end{equation}
We use $\E_t[\cdot]$ to denote the conditional expectation $\E[\,\cdot\mid\mathcal F_t]$ and define the conditional variance
\begin{equation}\label{eq:conditional-matrix-variance}
\var_t[Z]:=\E_t\bigl[\|Z-\E_t[Z]\|_F^2\bigr].
\end{equation}

Our main assumption is that the expected deviation of the practical update from its conceptual counterpart, defined as
\begin{equation}\label{eqn:practical-smsa-deviation}
\epsilon_t :=
    \bigl \|\E_t\bigl[\phi\bigl(\hM_t,\hW_t\bigr)\bigr] - \phi\bigl(M_t, W_t\bigr)\bigr\|_F,
\end{equation}
remains bounded almost surely for all $t\geq 0$.
In general, $\epsilon_t$ depends on the bias and variance of both estimators $\hM_t$ and $\hW_t$ as well as their numerical conditioning amplified by the blending map~$\phi$.
We will derive concrete bounds on~$\epsilon_t$ for Muon and ASGO in Sections~\ref{sec:muon-analysis} and~\ref{sec:asgo-analysis} respectively.

\begin{assumption}[Practical update deviation bound]\label{assum:bounded_err_smsa}
There exists a constant $\bar\epsilon\geq 0$ such that the practical update deviation $\epsilon_t$ defined in~\eqref{eqn:practical-smsa-deviation} satisfies
\[
\limsup_{t\to\infty}\epsilon_t\leq\bar\epsilon
\qquad\text{almost surely}.
\]
\end{assumption}

The following theorem states that the iterates generated by~\eqref{eqn:practical-smsa-dwd} converge almost surely to a neighborhood of $X_\ast$ (which is specified in Assumption~\ref{assum:aiming-smsa-dwd}) and the radius of the neighborhood is governed by~$\mu$ and~$\bar \epsilon$.
In fact, a persistent bias in the estimators can prevent almost-sure convergence to~$X_*$ itself.
We give a simple example in Appendix~\ref{sec:persistent-bias-example}.

\begin{theorem}[Convergence of practical algorithms]
\label{thm:practical-smsa-dwd-almost-sure}
Suppose that both Assumptions~\ref{assum:aiming-smsa-dwd} and~\ref{assum:bounded_err_smsa} hold, the initialization $X_0$ satisfies $\E\bigl[\|X_0\|_F^2\bigr]<\infty$, and the following conditions are satisfied:
\begin{enumerate}[label=(\alph*)]
\item Bounded variance: there is a constant~$b$ such that for all $t\geq 0$,
\begin{equation}\label{eq:estimator-smsa-dwd_bd_var}
\Var_t\bigl[\phi(\hM_t,\hW_t)\bigr] \leq b
\quad \text{almost surely;}
\end{equation}
\item The step sizes $\{\alpha_t\}$ are deterministic and satisfy
\[
\alpha_t\geq0,\qquad
\sum_{t=0}^\infty\alpha_t=\infty,
\qquad
\sum_{t=0}^\infty\alpha_t^2<\infty.
\]
\end{enumerate}
Then the iterates generated by~\eqref{eqn:practical-smsa-dwd} satisfy
\begin{equation}\label{eqn:explicit-neighborhood-radius}
\limsup_{t\to\infty}\|X_t-X_*\|_F
\leq\frac{\bar\epsilon}{\mu},
\qquad\text{almost surely}.
\end{equation}
\end{theorem}

The proof of Theorem~\ref{thm:practical-smsa-dwd-almost-sure} relies on the following extension of Dvoretzky's theorem~\cite{Dvoretzky1956}.

\begin{theorem}[An extension of Dvoretzky's theorem]
\label{thm:Dvoretzky}
Let $(\Omega,\cF,P)$ be a probability space with a filtration $(\mathcal F_t)_{t\geq0}$, and fix $X_*\in\R^{m\times n}$.
Let $X_t$ and $T_t(X_0, \ldots, X_t)$ be finite $\mathcal F_t$-measurable random variables in $\R^{m\times n}$ that satisfy, for every $t\geq0$,
\begin{equation}\label{eqn:Dvoretzky-alg}
X_{t+1}=T_t(X_0,\ldots,X_t)+Y_t,
\end{equation}
where $Y_t\in\R^{m\times n}$ is $\mathcal F_{t+1}$-measurable.
Suppose that, outside one null set of events, i.e., almost surely, the following inequality holds for all $t\geq t_0$ where $t_0$ is a deterministic integer,
\begin{equation}\label{eqn:Dvoretzky-ineq}
\|T_t(X_0, \ldots, X_t)-X_*\|_F^2
\leq\max\left\{a_t,(1+b_t)\|X_t-X_*\|_F^2+c_t-h_t\right\},
\end{equation}
where $a_t$, $b_t$, $c_t$ and $h_t$ be finite $\mathcal F_t$-measurable real-valued random variables.
Assume that $a_t$, $b_t$ and $c_t$ are nonnegative and $h_t\geq0$ for all sufficiently large $t$.
In addition, there exists a constant $a_\infty\in[0,\infty)$ such that outside the same null set,
\begin{equation}\label{eqn:Dvoretzky-conditions}
\limsup_{t\to\infty}\, a_t\,\leq\, a_\infty,
\qquad
\sum_{t=0}^\infty b_t<\infty,
\qquad
\sum_{t=0}^\infty c_t<\infty,
\qquad
\sum_{t=0}^\infty h_t=\infty.
\end{equation}
Finally, suppose
\begin{equation}\label{eqn:yt-squared-summable}
\sum_{t=0}^\infty\E\bigl[\|Y_t\|_F^2\bigr]<\infty,
\qquad\text{and}\quad
\E[Y_t\mid\mathcal F_t]=0
\quad\text{almost surely for every }t.
\end{equation}
Then
\[
\limsup_{t\to\infty}\|X_t-X_*\|_F^2\leq a_\infty
\quad\text{almost surely}.
\]
\end{theorem}

\begin{remark}
Dvoretzky's original theorem~\cite{Dvoretzky1956} uses three deterministic coefficient sequences $\{a_t\}, \{b_t\}, \{h_t\}$ without $\{c_t\}$ and assumes $a_t\to0$. In the same paper, Dvoretzky extends the same conclusion to filtration dependent coefficients $a_t(\mathcal{F}_t)$, $b_t(\mathcal{F}_t)$ and $h_t(\mathcal{F}_t)$. Derman and Sacks~\cite{DermanSacks1959} introduce an additional summable nonnegative term $\{c_t\}$ to the recursion and extend the theorem to multidimensional random vectors.
Venter \cite{Venter1966} further lifts the $a_t \to 0$ restriction, allowing $a_t=a >0$ to be a positive constant sequence. The present form in Theorem~\ref{thm:Dvoretzky} combines and generalizes these extensions.
In particular, we permit $\{a_t\}$ to be a time-varying sequence not necessarily converging to zero.
This was documented as a generalization in Dvoretzky's paper \cite{Dvoretzky1956} without proof. For complete reference, we provide a self-contained proof of Theorem~\ref{thm:Dvoretzky} in Appendix~\ref{sec:appendix_dvoretzky}.
\end{remark}

\begin{proof}[Proof of Theorem~\ref{thm:practical-smsa-dwd-almost-sure}]
First we can decompose the update in~\eqref{eqn:practical-smsa-dwd} as
\begin{align*}
    X_{t+1} &= (1-\alpha_t \lambda)X_t - \alpha_t \phi\bigl(\hM_t,\hW_t\bigr) \\
    &= (1-\alpha_t \lambda) X_t - \alpha_t \E_t\bigl[\phi\bigl(\hM_t,\hW_t\bigr)\bigr] + \alpha_t \left(\E_t\bigl[\phi\bigl(\hM_t,\hW_t\bigr)\bigr]-\phi\bigl(\hM_t,\hW_t\bigr)\right),
\end{align*}
which matches the formula~\eqref{eqn:Dvoretzky-alg} with
\begin{align*}
T_t(X_0, \ldots, X_t)&=(1-\alpha_t\lambda)X_t
-\alpha_t\E_t\bigl[\phi(\hM_t,\hW_t)\bigr],\\
Y_t&=\alpha_t\left(
\E_t\bigl[\phi(\hM_t,\hW_t)\bigr]-\phi(\hM_t,\hW_t)
\right).
\end{align*}
The above construction automatically implies
$\E[Y_t\mid\mathcal F_t]=0$ for every $t\geq 0$.
By the bounded variance assumption~\eqref{eq:estimator-smsa-dwd_bd_var} and the assumption that $\{\alpha_t^2\}$ is summable, we have
\begin{align*}
\sum_{t=0}^{\infty}\E\bigl[\|Y_t\|_F^2\bigr]
=\sum_{t=0}^{\infty}\alpha_t^2\, \Var_t\bigl[\phi(\hM_t,\hW_t)\bigr]
&\leq b\sum_{t=0}^{\infty}\alpha_t^2<\infty.
\end{align*}
Therefore, the conditions in~\eqref{eqn:yt-squared-summable} are satisfied.

The squared distance between $T_t(X_0, \ldots, X_t)$ and $X_*$ can be expanded as
\begin{align}
\bigl\|T_t(X_0, \ldots, X_t)-X_*\bigr\|_F^2
&=\left\|(1-\alpha_t\lambda)(X_t-X_*) -\alpha_t\left(\E_t[\phi(\hM_t,\hW_t)]+\lambda X_*\right) \right\|_F^2 \nonumber \\
&=(1-\alpha_t\lambda)^2\|X_t-X_*\|_F^2 ~+~ \alpha_t^2 \bigl\|\E_t[\phi(\hM_t,\hW_t)]+\lambda X_*\bigr\|_F^2 \nonumber \\
&\quad-2\alpha_t(1-\alpha_t\lambda) \left\langle X_t-X_*, \, \E_t[\phi(\hM_t,\hW_t)]+\lambda X_*\right\rangle.
\label{eqn:Dvoretzky-squared-dist}
\end{align}
For the cross term, we proceed by adding and subtracting $\phi\bigl(M_t, W_t\bigr)$ inside the inner product:
\begin{align}
& \left \langle \E_t \!\bigl[\phi\bigl(\hM_t,\hW_t\bigr)\bigr] + \lambda X_*, \, X_t-X_\ast \right \rangle \nonumber \\
=& \left \langle \E_t \!\bigl[\phi\bigl(\hM_t,\hW_t\bigr)\bigr] + \lambda X_*
   + \phi\bigl(M_t, W_t\bigr) - \phi\bigl(M_t, W_t\bigr), \, X_t-X_\ast \right \rangle \nonumber \\
=& \bigl\langle \phi\bigl(M_t, W_t\bigr) + \lambda X_*, \, X_t-X_*\bigr\rangle
    + \bigl\langle \E_t\bigl[\phi\bigl(\hM_t,\hW_t\bigr)\bigr]
    - \phi\bigl(M_t, W_t\bigr), \, X_t-X_* \bigr\rangle \nonumber\\
\geq& (\mu-\lambda)\|X_t-X_*\|_F^2 - \bigl\| \E_t\bigl[\phi\bigl(\hM_t,\hW_t\bigr)\bigr] - \phi\bigl(M_t, W_t\bigr)\bigr\|_F \norm{X_t-X_\ast}_F \nonumber \\
=& (\mu-\lambda)\|X_t-X_*\|_F^2 - \epsilon_t \norm{X_t-X_\ast}_F,
\label{eqn:Dvoretzky-cross-term}
\end{align}
where in the inequality we used the equivalent aiming condition~\eqref{eqn:aiming-smsa-dwd-x*} and the Cauchy-Shwartz inequality, and in the last equality we used the definition of~$\epsilon_t$ in~\eqref{eqn:practical-smsa-deviation}.

For the second squared term in~\eqref{eqn:Dvoretzky-squared-dist},
we apply the triangle inequality followed by Young's inequality:
\[
\bigl\|\E_t[\phi(\hM_t,\hW_t)]+\lambda X_*\bigr\|_F^2
\leq2\bigl\|\E_t[\phi(\hM_t,\hW_t)]\bigr\|_F^2 +2\lambda^2\|X_*\|_F^2.
\]
Next, we decompose the conditional mean into the conceptual update and its deviation to obtain
\begin{align*}
\bigl\|\E_t[\phi(\hM_t,\hW_t)]\bigr\|_F^2
&\leq2\|\phi(M_t,W_t)\|_F^2
+2\bigl\|\E_t[\phi(\hM_t,\hW_t)]-\phi(M_t,W_t)\bigr\|_F^2\\
&\leq2\left(L\|X_t-X_*\|_F+C\right)^2+2\epsilon_t^2\\
&\leq4L^2\|X_t-X_*\|_F^2+4C^2+2\epsilon_t^2,
\end{align*}
where we used the growth bound~\eqref{eqn:growth-smsa-dwd} and definition of~$\epsilon_t$.
Define the random variable
\begin{equation}\label{eqn:delta_t-def}
\delta_t:=8C^2+4\epsilon_t^2+2\lambda^2\|X_*\|_F^2.
\end{equation}
Then we have
\begin{equation}\label{eqn:Dvoretzky-square-term}
\bigl\|\E_t[\phi(\hM_t,\hW_t)]+\lambda X_*\bigr\|_F^2
\leq8L^2\|X_t-X_*\|_F^2+\delta_t.
\end{equation}

Since $\sum_{t=0}^\infty\alpha_t^2<\infty$, we have $\alpha_t\to0$.
Because the step sizes $\{\alpha_t\}$ are deterministic, there exists a deterministic $t_0$ such that $0\leq1-\alpha_t\lambda\leq1$ holds for all $t\geq t_0$.
For every $t\geq t_0$, substituting the bounds in~\eqref{eqn:Dvoretzky-cross-term} and~\eqref{eqn:Dvoretzky-square-term} into~\eqref{eqn:Dvoretzky-squared-dist} yields
\begin{align*}
\|T_t(X_0, \ldots, X_t)-X_*\|_F^2
&\leq(1-\alpha_t\lambda)^2\|X_t-X_*\|_F^2 -2\alpha_t(1-\alpha_t\lambda)(\mu-\lambda) \|X_t-X_*\|_F^2\\
&\quad+2\alpha_t(1-\alpha_t\lambda)\epsilon_t \|X_t-X_*\|_F +8\alpha_t^2L^2\|X_t-X_*\|_F^2+\alpha_t^2\delta_t\\
&\leq\left(1+\alpha_t^2(2\lambda\mu+8L^2)\right)\|X_t-X_*\|_F^2 - 2\alpha_t\mu\|X_t-X_F\|_F^2 \\
&\quad + 2\alpha_t\epsilon_t\|X_t-X_*\|_F+\alpha_t^2\delta_t.
\end{align*}
Young's inequality gives
\[
2\epsilon_t\|X_t-X_*\|_F
\leq\frac{\epsilon_t^2}{\mu}+\mu\|X_t-X_*\|_F^2.
\]
Consequently,
\[
\|T_t(X_0, \ldots, X_t)-X_*\|_F^2
\leq\left(1+\alpha_t^2(2\lambda\mu+8L^2)\right) \|X_t-X_*\|_F^2
-\alpha_t\left(\mu\|X_t-X_*\|_F^2
-\frac{\epsilon_t^2}{\mu}\right)
+\alpha_t^2\delta_t.
\]

Let $\bar\epsilon$ be the constant specified in Assumption~\ref{assum:bounded_err_smsa} and fix an arbitrary number $\eta>0$.
When $\norm{X_t - X_\ast}_F^2 > \bar\epsilon^2/\mu^2 + \eta$, we deduce from the above inequality that
\[
\norm{T_t(X_0, \ldots, X_t)-X_\ast}_F^2
\leq \bigl(1 + \alpha_t^2 (2 \lambda \mu + 8L^2 )\bigr)
\norm{X_t - X_\ast}_F^2
-\mu\alpha_t \left(\frac{\bar \epsilon^2}{\mu^2}+\eta
-\frac{\epsilon_t^2}{\mu^2}\right)
+\alpha_t^2 \delta_t.
\]
Otherwise, i.e., when $\norm{X_t - X_\ast}_F^2 \leq \bar\epsilon^2/\mu^2 + \eta$, we have
\[
\norm{T_t(X_0, \ldots, X_t)-X_\ast}_F^2
\leq \bigl(1 + \alpha_t^2 (2 \lambda \mu + 8L^2 )\bigr)
\left(\frac{\bar \epsilon^2}{\mu^2} + \eta\right)
+\frac{\alpha_t \epsilon_t^2}{\mu}+\alpha_t^2 \delta_t.
\]
We can combine these two cases as
\[
\norm{T_t(X_0, \ldots, X_t)-X_\ast}_F^2
\leq \max \left\{a_t,
(1+b_t)\norm{X_t-X_\ast}_F^2+c_t-h_t\right\},
\]
with the definitions
\[
a_t:=\bigl(1 + \alpha_t^2 (2 \lambda \mu + 8L^2 )\bigr)
\left(\frac{\bar \epsilon^2}{\mu^2} + \eta\right)
+\frac{\alpha_t \epsilon_t^2}{\mu}+\alpha_t^2 \delta_t,
\]
and
\[
b_t:=\alpha_t^2 (2 \lambda \mu + 8L^2),\qquad
c_t:=\alpha_t^2 \delta_t,\qquad
h_t:=\mu\alpha_t \left(\frac{\bar \epsilon^2}{\mu^2} +\eta -\frac{\epsilon_t^2}{\mu^2}\right).
\]
Thus~\eqref{eqn:Dvoretzky-ineq} holds for every $t\geq t_0$ where $t_0$ is the smallest $t$ such that $\alpha_t\leq1/\lambda$.

To apply Theorem~\ref{thm:Dvoretzky}, it remains to verify the conditions in~\eqref{eqn:Dvoretzky-conditions}.
Assumption~\ref{assum:bounded_err_smsa} implies that~$\epsilon_t$, and hence~$\delta_t$ through the definition in~\eqref{eqn:delta_t-def}, are bounded for sufficiently large~$t$.
Since $\alpha_t\to0$, we have
\[
a_\infty = \lim_{t\to\infty}a_t =\frac{\bar\epsilon^2}{\mu^2}+\eta < \infty.
\]
Moreover, we have
\[
\sum_{t=0}^\infty b_t
=(2\lambda\mu+8L^2)\sum_{t=0}^\infty\alpha_t^2<\infty,
\qquad
\sum_{t=0}^\infty c_t =\sum_{t=0}^\infty\alpha_t^2\delta_t<\infty \quad
\text{almost surely.}
\]
The assumption $\limsup_{t\to\infty}\epsilon_t\leq\bar\epsilon$ also implies that there exists a finite integer $t_\eta$ such that
\[
\frac{\epsilon_t^2}{\mu^2}
\leq\frac{\bar\epsilon^2}{\mu^2}+\frac{\eta}{2},
\qquad \forall\, t\geq t_\eta.
\]
Therefore, for every $t\geq t_\eta$,
\[
h_t=\mu\alpha_t \left(\frac{\bar \epsilon^2}{\mu^2} +\eta -\frac{\epsilon_t^2}{\mu^2}\right) \geq \frac{\mu\eta}{2}\alpha_t\geq0,
\]
and hence
\[
\sum_{t=t_\eta}^\infty h_t
\geq\frac{\mu\eta}{2}\sum_{t=t_\eta}^\infty\alpha_t
=\infty.
\]
The omitted prefix contains only finitely many bounded terms, therefore
$\sum_{t=0}^\infty h_t=\infty$ almost surely.
With all conditions satisfied, we can apply Theorem~\ref{thm:Dvoretzky} to obtain
\begin{equation*}
\limsup_{t\to\infty}\|X_t-X_*\|_F^2
\leq\frac{\bar\epsilon^2}{\mu^2}+\eta
\quad\text{almost surely}.
\end{equation*}
We apply this conclusion with $\eta=1/k$ for $k=1,2,3,\ldots$ and intersect the resulting countable family of probability-one events. As $k\to\infty$, this proves
\[
\limsup_{t\to\infty}\|X_t-X_*\|_F^2
\leq\frac{\bar\epsilon^2}{\mu^2}
\quad\text{almost surely}.
\]
Taking square roots on both sides of the above inequality gives~\eqref{eqn:explicit-neighborhood-radius}.
\end{proof}

Theorem~\ref{thm:practical-smsa-dwd-almost-sure} provides convergence guarantee for a broad class of practical second-moment SA methods.
To apply it to any particular method, we need to characterize the sequence $\epsilon_t$ defined in~\eqref{eqn:practical-smsa-deviation} under specific algorithmic settings.
In general, it is determined by the sensitivity of the blending function~$\phi$ under perturbation as well as the biases and variances of both the first- and second-moment estimators.
In the next two sections, we present concrete bounds on $\epsilon_t$ for Muon and ASGO respectively.

\section{Deviation bound for Muon}
\label{sec:muon-analysis}

We consider (the ideal version of) Muon with decoupled weight decay:
\begin{equation}\label{eq:practical_polar}
X_{t+1} = (1-\alpha_t \lambda) X_t - \alpha_t \polar(\hM_t),
\end{equation}
where $\hM_t \in \R^{m \times n}$ is an estimator of the first moment
$M(X_t)=\E_t[G(X_t,\xi_t)]$.
Recall that for a matrix $M$ of rank $r>0$ with compact SVD $M=U\Sigma V^T$ where $U\in\R^{m\times r}$ and $V\in\R^{n\times r}$, we have
\[
\polar(M)=UV^T.
\]
Throughout this section, we assume $m\leq n$ and the following assumption.
\begin{assumption}[Full-rank first-moment estimator]
\label{assum:polar-full-rank}
For every $t\geq 0$, we have $\E_t\bigl[\|\hM_t\|_F\bigr]<\infty$ almost surely and
\[
\Pr\bigl(\operatorname{rank}(\hM_t)=m\mid\mathcal F_t\bigr)=1,
\qquad
\operatorname{rank}(\E_t[\hM_t])
=\operatorname{rank}(M_t)=m
\quad\text{almost surely}.
\]
\end{assumption}

Under the above assumption, we have the following bound on the deviation defined in~\eqref{eqn:practical-smsa-deviation}.

\begin{theorem}[Deviation bound for Muon]
\label{thm:overall_polar_err_bound}
Suppose Assumption~\ref{assum:polar-full-rank} holds and $\E_t\bigl[\|\hM_t\|_F^2\bigr]<\infty$ almost surely for every $t\geq 0$.
Then the following holds almost surely for every $t\geq 0$:
\begin{equation}\label{eq:overall_polar_err_bound}
\epsilon_t:=\bigl\|\E_t[\polar(\hM_t)]-\polar(M_t)\bigr\|_F ~\leq~
\frac{3}{\hat\sigma_t}\sqrt{\var_t[\hM_t]} +\frac{4}{\hat\sigma_t+\sigma_t} \bigl\|\E_t[\hM_t]-M_t\bigr\|_F,
\end{equation}
where $\hat\sigma_t$ and $\sigma_t$ are the smallest positive singular values of $\E_t[\hM_t]$ and $M_t$ respectively.
\end{theorem}

The bound in~\eqref{eq:overall_polar_err_bound} accounts for the estimator’s bias and variance, and the numerical conditioning of both $M_t$ and $\E_t[\hM_t]$ through their smallest singular values.
According to Theorem~\ref{thm:practical-smsa-dwd-almost-sure}, its asymptotic value (more precisely limit superior) of~$\epsilon_t$ bound the size of the neighborhood around~$X_*$ where the algorithm~\eqref{eq:practical_polar} converges to.
The key to the proof of Theorem~\ref{thm:overall_polar_err_bound} hinges on the sensitivity of the mapping $\ortho(\cdot)$ to perturbations in the input matrix. The following lemma characterizes this stability in terms of unitarily invariant norms $\vertiii{\cdot}$.

\begin{lemma}[{Wide matrix variant of~\cite[][Theorem~2]{li1995new}}]
\label{lem:polar_sensitivity_wide_mat}
Suppse $m\leq n$ and $A,\tilde A\in\R^{m\times n}$ have full row rank.
Let $Q=\polar(A)$, $\tilde Q=\polar(\tilde A)$, and $\sigma_m$ and $\tilde \sigma_m$ be the smallest singular values of $A$ and $\tilde A$ respectively. Then for any unitarily invariant norm $\vertiii{\cdot}$,
    \begin{equation}\label{eq:polar_sensitivity_wide_mat}
        \vertiii{Q - \tilde Q} \leq \left(\frac{2}{\sigma_m + \tilde \sigma_m} + \frac{1}{\max\{\sigma_m, \tilde \sigma_m\}}\right) \vertiii{A - \tilde A}.
    \end{equation}
\end{lemma}

Lemma~\ref{lem:polar_sensitivity_wide_mat} establishes that the stability of $\ortho(\cdot)$ is governed by the input error $\vertiii{A-\tilde A}$, scaled inversely by the inputs' smallest singular values. Consequently, well-conditioned matrices yield stable orthogonal factors, whereas near-rank deficiency severely amplifies perturbation errors.
This lemma is the wide matrix variant of Theorem 2 in~\cite{li1995new}, which assumes that $A$ and $\tilde A$ are tall matrices with $m \geq n$. The proofs for the two cases are similar so we omit the one here.

\bigskip

In order to derive the bound~\eqref{eq:overall_polar_err_bound}, we first decompose the deviation~$\epsilon_t$ as follows:
\begin{align}
\epsilon_t :=& \bigl\|\E_t[\polar(\hM_t)] - \polar(M_t)\bigr\|_F \nonumber\\
=& \bigl\|\E_t[\polar(\hM_t)] - \polar(\E_t[\hM_t]) + \polar(\E_t[\hM_t]) - \polar(M_t)\bigr\|_F \nonumber\\
\leq& \bigl\|\E_t[\polar(\hM_t)] - \polar(\E_t[\hM_t])\bigr\|_F + \bigl\|\polar(\E_t[\hM_t]) - \polar(M_t)\bigr\|_F.
\label{eqn:ortho-deviation-decomposition}
\end{align}
We bound the two terms separately.
The following lemma shows that the first term captures the error induced by the variance of the estimator $\hM_t$, amplified by the conditioning of $\ortho(\cdot)$.

\begin{lemma}\label{lem:err_polar[E]-E[polar]}
Suppose Assumption~\ref{assum:polar-full-rank} holds and
$\E_t\bigl[\|\hM_t\|_F^2\bigr]<\infty$ almost surely for all $t\geq 0$.
Then,
\[
\bigl\|\E_t[\polar(\hM_t)]-\polar(\E_t[\hM_t])\bigr\|_F
\leq\frac{3}{\hat\sigma_t}\sqrt{\var_t[\hM_t]}
\quad \text{almost surely.}
\]
where $\hat\sigma_t$ is the smallest positive singular values of $\E_t[\hM_t]$.
\end{lemma}

\begin{proof}
First, we use the linearity of expectation and Jensen's inequality to obtain
\begin{align*}
    \bigl\|\E_t[\polar(\hM_t)] - \polar(\E_t[\hM_t])\bigr\|_F
    &= \bigl\|\E_t[\polar(\hM_t) - \polar(\E_t[\hM_t])]\bigr\|_F \\
    &\leq \E_t\bigl[\bigl\|\polar(\hM_t) - \polar(\E_t[\hM_t])\bigr\|_F\bigr].
\end{align*}
Assumption~\ref{assum:polar-full-rank} allows us to apply Lemma~\ref{lem:polar_sensitivity_wide_mat} so that almost surely for every $t\geq 0$:
\begin{align*}
    \E_t\bigl[\bigl\|\polar(\hM_t)-\polar(\E_t[\hM_t])\bigr\|_F\bigr]
    &\leq\E_t\left[\left(
    \frac{2}{\tilde\sigma_t+\hat\sigma_t}
    +\frac{1}{\max\{\tilde\sigma_t,\hat\sigma_t\}}
    \right)\bigl\|\hM_t-\E_t[\hM_t]\bigr\|_F\right]\\
    &\leq\frac{3}{\hat\sigma_t} \E_t\bigl[\bigl\|\hM_t-\E_t[\hM_t]\bigr\|_F\bigr],
\end{align*}
where $\tilde\sigma_t$ and $\hat\sigma_t$ are the smallest positive singular values of $\hM_t$ and $\E_t[\hM_t]$ respectively.
To further simplify the bound, recall that Frobenius norm is the square-root of the sum of the squares of all elements, which allows us to rewrite the upper bound as follows:
\[
    \frac{3}{\hat \sigma_t} \E_t \bigl[\bigl\|\hM_t - \E_t[\hM_t]\bigr\|_F\bigr] = \frac{3}{\hat \sigma_t} \E_t \left[\sqrt{\textstyle\sum_{i,j} \bigl(\hM_t(i,j) - \E_t\bigl[\hM_t(i,j)\bigr]\bigr)^2}\right].
\]
Applying Jensen's inequality again to interchange expectation and the square root, we obtain:
\begin{align*}
    \frac{3}{\hat \sigma_t} \E_t \left[\sqrt{\textstyle\sum_{i,j} \bigl(\hM_t(i,j) - \E_t[\hM_t(i,j)]\bigr)^2}\right]
    &\leq \frac{3}{\hat \sigma_t} \sqrt{\textstyle\sum_{i,j} \E_t\bigl[\bigl(\hM_t(i,j) - \E_t[\hM_t(i,j)]\bigr)^2\bigr]}\\
    &= \frac{3}{\hat \sigma_t} \sqrt{\var_t[\hM_t]}.
\end{align*}
Combining the chain of inequalities, we obtain the desired result.
\end{proof}

The second term on the right-hand side of~\eqref{eqn:ortho-deviation-decomposition}
accounts for the perturbation of $\ortho(\cdot)$ due to the bias of $\hM_t$ relative to~$M_t$.
\begin{lemma}\label{lem:err_polar(EU)-polar(ED)}
Under Assumption~\ref{assum:polar-full-rank}, the following error bound holds almost surely for every $t$:
\[
        \bigl\|\polar(\E_t[\hM_t]) - \polar(M_t)\bigr\|_F \leq \frac{4}{\hat \sigma_t+\sigma_t} \bigl\|\E_t[\hM_t] - M_t\bigr\|_F.
\]
where $\hat\sigma_t$ and $\sigma_t$ are the smallest positive singular values of $\E_t[\hM_t]$ and $M_t$ respectively.
\end{lemma}

\begin{proof}
In this case, we can apply Lemma~\ref{lem:polar_sensitivity_wide_mat} directly to obtain
\begin{align*}
    \bigl\|\polar(\E_t[\hM_t]) - \polar(M_t)\bigr\|_F&\leq \left(\frac{2}{\hat \sigma_t+\sigma_t} + \frac{1}{\max\{\hat \sigma_t, \sigma_t\}}\right) \bigl\|\E_t[\hM_t] - M_t\bigr\|_F.
\end{align*}
Moreover, since $\max\{\hat \sigma_t, \sigma_t\} \geq \frac{1}{2}(\hat \sigma_t + \sigma_t)$, the coefficient can be further upper bounded as
\[
    \left(\frac{2}{\hat \sigma_t+\sigma_t} + \frac{1}{\max\{\hat \sigma_t, \sigma_t\}}\right) \leq \frac{4}{\hat \sigma_t+\sigma_t},
\]
which gives the desired result.
\end{proof}

Theorem~\ref{thm:overall_polar_err_bound} is a direct consequence of \eqref{eqn:ortho-deviation-decomposition} and the combination of Lemma~\ref{lem:err_polar[E]-E[polar]} and~\ref{lem:err_polar(EU)-polar(ED)}.

\bigskip

In order to apply Theorem~\ref{thm:practical-smsa-dwd-almost-sure} to Muon, we need to check Assumptions~\ref{assum:aiming-smsa-dwd} and~\ref{assum:bounded_err_smsa} as well as the bounded variance condition~\eqref{eq:estimator-smsa-dwd_bd_var}.
For Assumption~\ref{assum:aiming-smsa-dwd}, we gave an examples of $M(X)$ that satisfies the aiming condition at the end of Section~\ref{sec:aiming_general-smsa} and provide empirical evidences for deep learning in Section~\ref{sec:experiments}.
In addition, the fact $\|\ortho(M(X))\|_F\leq \sqrt{m}$ implies that the growth condition~\eqref{eqn:growth-smsa-dwd} holds with $L=0$ and $C=\sqrt{m}$.
For Assumption~\ref{assum:bounded_err_smsa}, it is sufficient to assume that the first-moment estimator $\hM_t$ has bounded bias and variance, then Theorem~\ref{thm:overall_polar_err_bound} implies that the expected deviation~$\epsilon_t$ is bounded almost surely.
Finally, we have
\[
\Var_t\bigl[\polar(\hM_t)\bigr] \leq \E_t\bigl[\|\polar(\hM_t)\|_F^2\bigr] \leq m,
\]
which means that the bounded variance condition~\eqref{eq:estimator-smsa-dwd_bd_var} holds with $b=m$.

\section{Deviation bound for ASGO}
\label{sec:asgo-analysis}

We consider ASGO \cite{an2025asgo} with decoupled weight decay:
\begin{equation}\label{eq:practical_sqrtinv}
     X_{t+1} = (1 - \lambda \alpha_t )X_t - \alpha_t \hW_t^{-1/2} \hM_t,
\end{equation}
where $\hM_t$ is an estimator for the first-moment $M_t= \E_t[G_t]$ and $\hW_t$ is an estimator for the second moment $W_t=\E_t[G_t G_t^T]$.
We assume that both the second-moment $W_t$ and its estimator $\hW_t$ are positive definite with uniform lower bounds on their eigenvalues.

\begin{assumption}[Positive-definiteness of the second moment]
\label{assum:sqrtinv-min-eig}
The second-moment estimator $\hW_t$ is symmetric and positive definite for all $t\geq 0$.
There are deterministic constants $\lambda_m,\hat\lambda_m>0$ such that
$W_t\succeq\lambda_m I$ and $\hW_t\succeq\hat\lambda_m I$
hold with probability one for every $t\geq 0$.
\end{assumption}

For ASGO, the deviation defined in~\eqref{eqn:practical-smsa-deviation} specializes to
$\epsilon_t = \bigl\|\E_t[\hW_t^{-1/2} \hM_t] - W_t^{-1/2} M_t \bigr\|_F$, for which we have the following bound.

\begin{theorem}[Deviation bound for ASGO]
\label{thm:overall_asgo_err_bound}
Suppose Assumption~\ref{assum:sqrtinv-min-eig} holds and
$\E_t\bigl[\|\hM_t\|_F^2\bigr]<\infty$ almost surely.
Then the following bound holds almost surely for each $t\geq 0$:
\begin{align}
\epsilon_t
&~\leq~ \frac{\hat\lambda_m^{-3/2}}{2} \sqrt{\var_t\bigl[\hW_t\bigr]} \left(\bigl\|\E_t [\hM_t]\bigr\|_F + \sqrt{\var_t[\hM_t]}\right)
\label{eq:sqrtinv_err_bd} \\
&\quad~ + \bigl(\lambda_m \hat \lambda_m^{1/2} + \hat \lambda_m \lambda_m^{1/2}\bigr)^{-1} \bigl\|\E_t[\hW_t] - W_t\bigr\|_F \bigl\|\E_t[\hM_t]\bigr\|_2 ~+~ \lambda_m^{-1/2} \bigl\|\E_t[\hM_t] - M_t\bigr\|_F,
\nonumber
\end{align}
where $\hat\lambda_m$ and $\lambda_m$ are the uniform lower bound on the eigenvalues of $\hW_t$ and $W_t$ respectively.
\end{theorem}

Theorem~\ref{thm:overall_asgo_err_bound} establishes a deviation bound that is governed by the bias, variance, and numerical conditioning of the estimators $\hM_t$ and $\hW_t$.
The following lemma captures the numerical conditioning of the matrix square-root inverse, as part of the blending function $\phi(M,W)=W^{-1/2}M$.

\begin{lemma}\label{lem:mat_invsqrt_uniform_perturbation_bound}
Let $A,B\in\R^{m\times m}$ be symmetric positive definite matrices with smallest eigenvalues $\lambda_A>0$ and $\lambda_B>0$ respectively. Then
\begin{equation}\label{eq:mat_invsqrt_uniform_perturbation_bound}
\bigl\|A^{-1/2}-B^{-1/2}\bigr\|_F \leq\bigl(\lambda_A\lambda_B^{1/2}+\lambda_B\lambda_A^{1/2}\bigr)^{-1}\|B-A\|_F.
\end{equation}
\end{lemma}

The proof of Lemma~\ref{lem:mat_invsqrt_uniform_perturbation_bound} is given in Appendix~\ref{sec:sqrtinv-perturb-bound-proof}.
To prove Theorem~\ref{thm:overall_asgo_err_bound}, we follow a similar strategy for Muon by first decomposing the deviation~$\epsilon_t$ as follows:
\begin{align}
\epsilon_t
:=& \bigl\|\E_t\bigl[\hW_t^{-1/2} \hM_t\bigr] - W_t^{-1/2} M_t \bigr\|_F \nonumber\\
=&\bigl\|\E_t\bigl[\hW_t^{-1/2} \hM_t\bigr] - \E_t\bigl[\hW_t \bigr]^{-1/2} \E_t[\hM_t] + \E_t[\hW_t]^{-1/2} \E_t[\hM_t] - W_t^{-1/2} M_t \bigr\|_F \nonumber\\
\leq& \bigl\|\E_t[\hW_t^{-1/2} \hM_t] - \E_t[\hW_t]^{-1/2} \E_t[\hM_t]\bigr\|_F + \bigl\|\E_t[\hW_t]^{-1/2} \E_t[\hM_t] - W_t^{-1/2} M_t \bigr\|_F.
\label{eqn:asgo-deviation-decomposition}
\end{align}
We bound these two terms separately through the next two lemmas.

\begin{lemma}\label{lem:EM_EW_invsqrt-EMW_invsqrt}
Suppose Assumption~\ref{assum:sqrtinv-min-eig} holds and
$\E_t\bigl[\|\hM_t\|_F^2\bigr]<\infty$
and $\E_t\bigl[\|\hW_t\|_F^2\bigr]<\infty$
almost surely for every $t\geq 0$.
Then we have for every $t\geq 0$, almost surely,
\[
\begin{aligned}
    &\bigl\|\E_t[\hW_t^{-1/2} \hM_t] - \E_t[\hW_t]^{-1/2} \E_t[\hM_t]\bigr\|_F \leq \frac{\hat\lambda_m^{-3/2}}{2} \sqrt{\var_t\bigl[\hW_t\bigr]} \left(\bigl\|\E_t [\hM_t]\bigr\|_F + \sqrt{\var_t[\hM_t]}\right),
\end{aligned}
\]
where $\hat\lambda_m$ is the uniform lower bound on the eigenvalue of $\hW_t$ in Assumption~\ref{assum:sqrtinv-min-eig}.
\end{lemma}
\begin{proof}
First, by definition of the conditional expectation, we have
\[
\E_t[\hW_t]^{-1/2} \E_t[\hM_t] = \E_t[\E_t[\hW_t]^{-1/2} \hM_t],
\]
which allows us to rewrite the error term into a form that can apply Jensen's inequality:
\begin{align*}
\bigl\|{\E_t[\hW_t^{-1/2} \hM_t] - \E_t[\hW_t]^{-1/2} \E_t[\hM_t]}\bigr\|_F
&= \bigl\|\E_t\bigl[ \bigl(\hW_t^{-1/2} - \E_t[\hW_t]^{-1/2}\bigr) \hM_t\bigr]\bigr\|_F \\
&\leq \E_t\bigl[\bigl\|\bigl(\hW_t^{-1/2} - \E_t[\hW_t]^{-1/2}\bigr)\hM_t \bigr\|_F\bigr].
\end{align*}
Since the Frobenius norm is submultiplicative, we can further upper bound the norm of a matrix product by the product of matrix norms and then apply the Cauchy-Schwarz inequality:
\begin{align}
\E_t\bigl[\bigl\|\bigl(\hW_t^{-1/2} - \E_t[\hW_t]^{-1/2}\bigr)\hM_t \bigr\|_F\bigr]
&\leq \E_t\bigl[\bigl \|\hW_t^{-1/2} - \E_t[\hW_t]^{-1/2}\bigr \|_F \bigl\|\hM_t\bigr\|_F \bigr] \nonumber \\
&\leq \sqrt{\E_t\bigl[\bigl\|\hW_t^{-1/2} - \E_t[\hW_t]^{-1/2}\bigr\|_F^2\bigr]} \sqrt{\E_t\bigl[\bigl\|\hM_t\bigr\|_F^2\bigr]} .
\label{eq:EM_EW_invsqrt-EMW_invsqrt_0}
\end{align}
Next we bound the two expected squared Frobenius norms separately.

For $\E_t\bigl[\bigl\|{\hW_t^{-1/2}\!-\E_t[\hW_t]^{-1/2}}\bigr\|_F^2\bigr]$, we apply Lemma~\ref{lem:mat_invsqrt_uniform_perturbation_bound} together with the fact $\lambda_m\geq\hat\lambda_m$ to obtain
\begin{equation}\label{eq:EM_EW_invsqrt-EMW_invsqrt:E_W_invsqrt_final}
\begin{aligned}
    \E_t\bigl[\bigl\|{\hW_t^{-1/2} - \E_t[\hW_t]^{-1/2}}\bigr\|_F^2\bigr]
    &\leq \E_t\!\left[\Bigl(\frac{1}{2}\hat\lambda_m^{-3/2} \bigl\|\hW_t - \E_t[\hW_t]\bigr\|_F \Bigr)^2\right]\\
    &=\frac{\hat\lambda_m^{-3}}{4}  \E_t\!\left[\bigl\|\hW_t - \E_t[\hW_t]\bigr\|_F^2\right]\\
    &=\frac{\hat\lambda_m^{-3}}{4}\var_t[\hW_t].
\end{aligned}
\end{equation}
For the second-moment term $\E_t\bigl[\bigl\|\hM_t\bigr\|_F^2\bigr]$, we have the mean-variance decomposition
\begin{equation}\label{eq:EM_EW_invsqrt-EMW_invsqrt:E_M_sq}
    \E_t\bigl[\|\hM_t\|_F^2\bigr]
    =\|\E_t[\hM_t]\|_F^2+\var_t[\hM_t].
\end{equation}
Substituting the identity~\eqref{eq:EM_EW_invsqrt-EMW_invsqrt:E_M_sq}
    and the bound~\eqref{eq:EM_EW_invsqrt-EMW_invsqrt:E_W_invsqrt_final}
    into~\eqref{eq:EM_EW_invsqrt-EMW_invsqrt_0}, we get
\begin{align*}
\bigl\|\E_t[\hW_t^{-1/2} \hM_t] - \E_t[\hW_t]^{-1/2} \E_t[\hM_t]\bigr\|_F
&\leq \sqrt{\frac{\hat\lambda_m^{-3}}{4}\var_t[\hW_t]}
        \sqrt{\bigl\|\E_t[\hM_t]\bigr\|_F^2+\var_t[\hM_t]} \\
&\leq \frac{\hat\lambda_m^{-3/2}}{2} \sqrt{\var_t\bigl[\hW_t\bigr]} \left(\bigl\|\E_t [\hM_t]\bigr\|_F + \sqrt{\var_t[\hM_t]}\right),
\end{align*}
where the last inequality follows from sub-additivity of the square-root function.
\end{proof}

The next lemma bounds the second term on the right side of~\eqref{eqn:asgo-deviation-decomposition}.

\begin{lemma}\label{lem:EM_EW_invsqrt-ED_EDD_invsqrt}
Suppose Assumption~\ref{assum:sqrtinv-min-eig} holds and $\E_t\bigl[\|\hW_t\|_F\bigr]<\infty$ and $\E_t\bigl[\|\hM_t\|_F\bigr]<\infty$ for every $t\geq 0$.
Then, almost surely for every $t\geq 0$,
\begin{align*}
    &\bigl\|\E_t[\hW_t]^{-1/2} \E_t[\hM_t] - W_t^{-1/2} M_t\bigr\|_F\\
\leq~& \bigl(\lambda_m \hat \lambda_m^{1/2} + \hat \lambda_m \lambda_m^{1/2}\bigr)^{-1} \bigl\|\E_t[\hW_t] - W_t\bigr\|_F \bigl\|\E_t[\hM_t]\bigr\|_2  ~+~ \lambda_m^{-1/2} \bigl\|\E_t[\hM_t] - M_t\bigr\|_F,
\end{align*}
where $\lambda_m$ and $\hat\lambda_m$ are the uniform lower bounds for the eigenvalues of $W_t$ and $\hW_t$ respectively.
\end{lemma}

\begin{proof}
Adding and subtracting $W_t^{-1/2} \E_t[\hM_t] $ inside the norm and rearranging the terms, we get
\begin{align*}
    &\bigl\|\E_t[\hW_t]^{-1/2} \E_t[\hM_t] - W_t^{-1/2} M_t\bigr\|_F\\
    =~&\bigl\|\E_t[\hW_t]^{-1/2} \E_t[\hM_t] - W_t^{-1/2} \E_t[\hM_t]  + W_t^{-1/2} \E_t[\hM_t]  - W_t^{-1/2} M_t\bigr\|_F\\
    =~&\bigl\|\bigl(\E_t[\hW_t]^{-1/2} - W_t^{-1/2}\bigr) \E_t[\hM_t] + W_t^{-1/2} \bigl(\E_t[\hM_t] - M_t\bigr)\bigr\|_F.
\end{align*}
Applying the triangle inequality followed by the submultiplicativity of matrix norms, we get
\begin{align}
    &\bigl\|\bigl(\E_t[\hW_t]^{-1/2} - W_t^{-1/2}\bigr) \E_t[\hM_t] + W_t^{-1/2} \bigl(\E_t[\hM_t] - M_t\bigr)\bigr\|_F\nonumber\\
\leq~& \bigl\|\bigl(\E_t[\hW_t]^{-1/2} - W_t^{-1/2}\bigr) \E_t[\hM_t] \bigr\|_F + \bigl\|W_t^{-1/2} \bigl(\E_t[\hM_t] - M_t\bigr)\bigr\|_F\nonumber\\
\leq~& \bigl\|\E_t[\hW_t]^{-1/2} - W_t^{-1/2}\bigr\|_F \bigl\|\E_t[\hM_t]\bigr\|_2  + \bigl\|W_t^{-1/2}\bigr\|_2 \bigl\|\E_t[\hM_t] - M_t\bigr\|_F.
\label{eq:EM_EW_invsqrt-ED_EDD_invsqrt_0}
\end{align}
Again we bound the two resulting terms separately.
For the first one, we use Lemma~\ref{lem:mat_invsqrt_uniform_perturbation_bound} to obtain
\[
\bigl\|\E_t[\hW_t]^{-1/2} - W_t^{-1/2}\bigr\|_F
\leq \bigl(\lambda_m \hat \lambda_m^{1/2} + \hat \lambda_m \lambda_m^{1/2}\bigr)^{-1} \bigl\|\E_t[\hW_t] - W_t\bigr\|_F
\]
For the second term on the right-hand side of~\eqref{eq:EM_EW_invsqrt-ED_EDD_invsqrt_0}, we have
\[
    \bigl\|W_t^{-1/2}\bigr\|_2 =  \lambda_{\min}\bigl(W_t\bigr)^{-1/2} \leq \lambda_m^{-1/2}.
\]
Substituting the above two bounds into~\eqref{eq:EM_EW_invsqrt-ED_EDD_invsqrt_0}, we get
\begin{align*}
    &\bigl\|\E_t[\hW_t]^{-1/2} - W_t^{-1/2}\bigr\|_F \bigl\|\E_t[\hM_t]\bigr\|_2 + \bigl\|W_t^{-1/2}\bigr\|_2 \bigl\|\E_t[\hM_t] - M_t\bigr\|_F \\
\leq~& \bigl(\lambda_m \hat \lambda_m^{1/2} + \hat \lambda_m \lambda_m^{1/2}\bigr)^{-1} \bigl\|\E_t[\hW_t] - W_t\bigr\|_F \bigl\|\E_t[\hM_t]\bigr\|_2  + \lambda_m^{-1/2} \bigl\|\E_t[\hM_t] - M_t\bigr\|_F.
\end{align*}
This finishes the proof.
\end{proof}

Theorem~\ref{thm:overall_asgo_err_bound} is a direct consequence of \eqref{eqn:asgo-deviation-decomposition} and the combination of
Lemmas~\ref{lem:EM_EW_invsqrt-EMW_invsqrt}
and~\ref{lem:EM_EW_invsqrt-ED_EDD_invsqrt}.

\bigskip

In order to apply Theorem~\ref{thm:practical-smsa-dwd-almost-sure} to ASGO, we need Assumption~\ref{assum:bounded_err_smsa} to hold, which requires Assumption~\ref{assum:sqrtinv-min-eig} as well as $\E_t\bigl[\|\hM_t\|_F^2\bigr]$ being bounded.
The uniform lower bound on the eigenvalues in Assumption~\ref{assum:sqrtinv-min-eig} is not a limiting restriction because most practical algorithms add a small offset~$\epsilon I$ to their second-moment estimators.
For Assumption~\ref{assum:aiming-smsa-dwd}, the growth condition~\eqref{eqn:growth-smsa-dwd} holds with $L=0$ and $C=\sqrt{m}$ because
\[
\bigl\|W_t^{-1/2}M_t\bigr\|_F^2 = \trace\left(W_t^{-1}M_t M_t^T \right) = \trace\bigl(\E_t[G_t G_t^T]^{-1} \E_t[G_t] \E_t[G_t]^T\bigr) \leq m.
\]
We provide empirical evidence for the aiming condition~\eqref{eqn:aiming-smsa-dwd} in Section~\ref{sec:experiments}.

\section{Numerical experiments}\label{sec:experiments}

In this section, we present numerical experiments to illustrate some characteristics of second-moment SA methods captured by our theory.
We focus on training a large language model using ASGO~\cite{an2025asgo} and Muon~\cite{jordan2024muon}, listed in the equations~\eqref{eqn:asgo} and~\eqref{eqn:muon-newtonschulz} respectively.
Both methods can be interpreted through the optimal fused preconditioner in Section~\ref{sec:fused-precond}, but with simplifications assuming different noise structures, presented in Sections~\ref{sec:derive-muon} and~\ref{sec:derive-asgo} respectively.

\begin{figure}[t]
    \centering
    \includegraphics[width=1.0\linewidth]{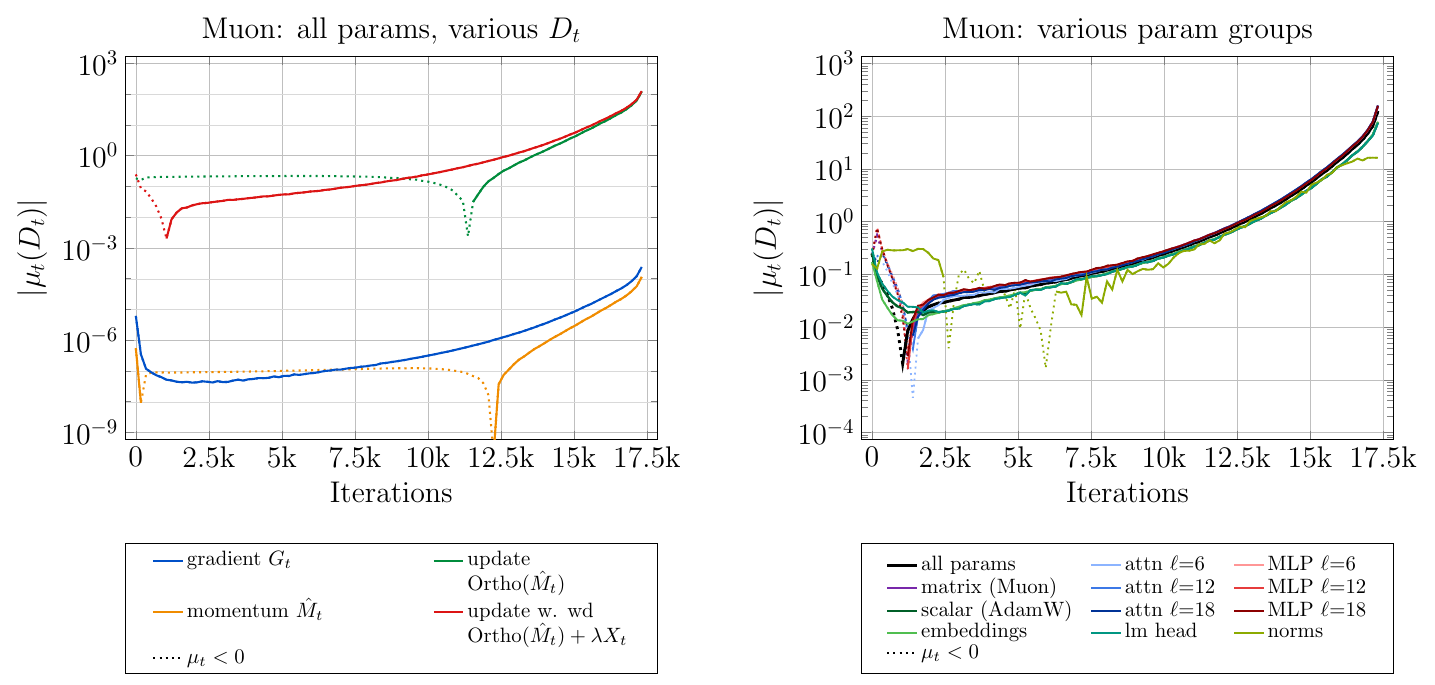}
    \caption{Empirical evaluation of the aiming condition for training a large language model with Muon $(\beta_1=0.95)$.
    The $y$ axis shows the absolute values of $\mu_t(D_t)$ at each iteration.
    Left: $\mu_t$ computed for the whole group with various directions
    $D_t=G_t$, $\hM_t$, $\polar(\hM_t)$ and $\polar(\hM_t)+\lambda X_t$.
    Right: $\mu_t$ computed for $D_t=\polar(\hM_t)+\lambda X_t$ but over various parameter groups.}
    \label{fig:llama_polar_aiming_mu}
\end{figure}

All experiments are conducted to train a LLaMA-style model \cite{meta2024llama3herdmodels} with 1.81 Billion parameters (with bf16 precision).
The model has $25$ layers, $2048$ hidden dimensions, $16$ attention heads, $8/3$ MLP intermediate expansion and $4096$ sequence length on the FineWeb-Edu dataset \cite{penedo2024fineweb}.
Training is done with a batch size of 2 million (2048K) tokens, distributed across 32 NVIDIA H100 GPUs, and run for 17.5K steps to match a Chinchilla-optimal 20 tokens-per-parameter (TPP) \cite{hoffmann2022chinchinlla}.

We apply Muon and ASGO to optimize the matrix parameters within the model.
Non-matrix parameters are optimized using AdamW ($\alpha = 0.003$, weight decay $\lambda = 0.1$, $\beta_1 = 0.9$, $\beta_2 = 0.95$). We employ a standard cosine decay learning rate schedule with a 1\% linear warmup, decaying to 1\% of the peak learning rate $\alpha$ at the final step. All optimal hyperparameter configurations are evaluated across three independent random seeds governing model parameter initialization and data loading using the PyTorch package\cite{paszke2019pytorch}.

\subsection{Empirical evidence of the aiming condition}

\begin{figure}[t]
    \centering
    \includegraphics[width=1.0\linewidth]{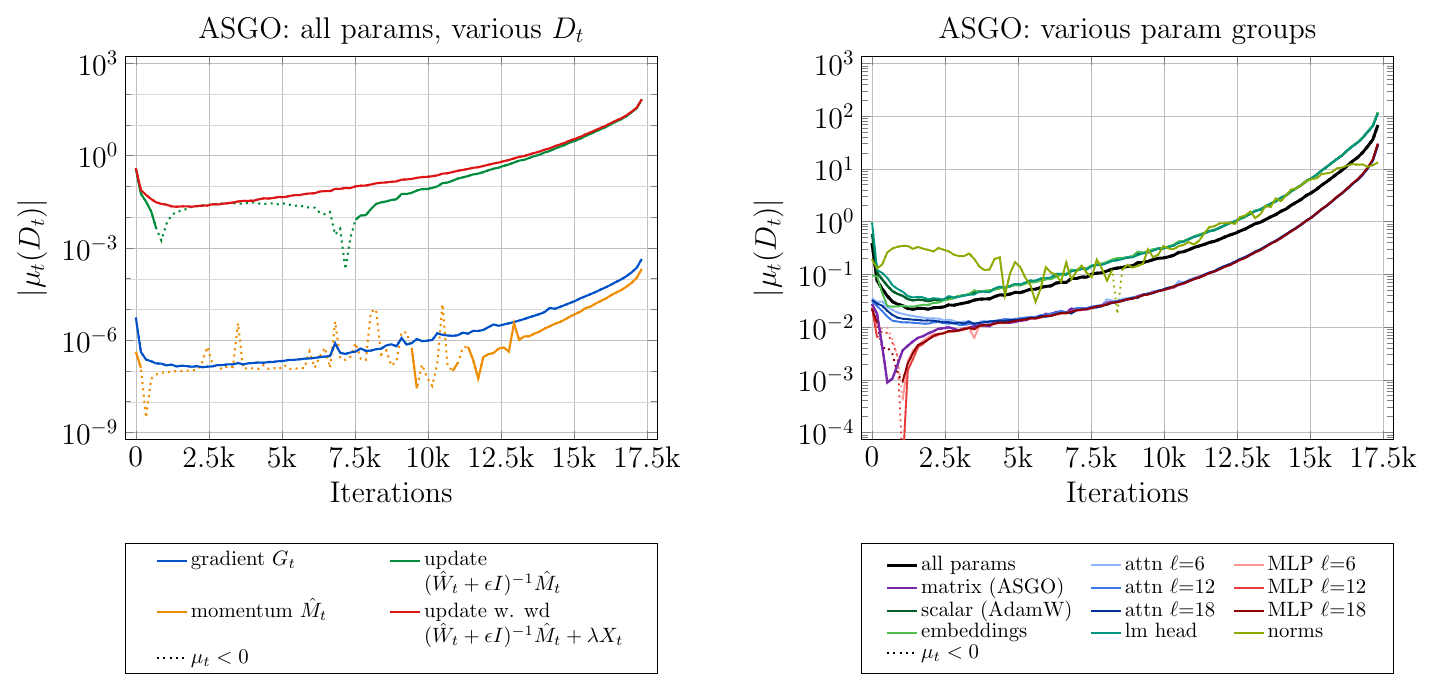}
    \caption{Empirical evaluation of the aiming condition for training a large language model with ASGO $(\beta_1=0.9, \, \beta_2=0.95)$. The $y$ axis shows the absolute values of $\mu_t(D_t)$.
    Left: $\mu_t$ computed for the whole group with various directions
    $D_t=G_t$, $\hM_t$, $\hW_t^{-1/2} \hM_t$ and $\hW_t^{-1/2} \hM_t+\lambda X_t$.
    Right: $\mu_t$ computed for $D_t=\hW_t^{-1/2} \hM_t+\lambda X_t$ but over various parameter groups.
}
    \label{fig:llama_asgo_aiming_mu}
\end{figure}

Our first set of experiments demonstrate that the aiming condition~\eqref{eqn:aiming-smsa-dwd} holds empirically during most of the training process.
For each training trajectory, we use the final iterate $X_\text{final}$ to approximate $X_*$, and record the the inner product between $\phi(\hM_t,\hW_t)+\lambda X_t$ and $X_t-X_{\text{final}}$.
We then normalize the inner product by squared norm of $X_t - X_{\text{final}}$ to obtain the empirical estimate of the parameter~$\mu$ in~\eqref{eqn:aiming-smsa-dwd}.
More generally, we define
$$\mu_t(D_t) :=\frac{\langle D_t,\, X_t-X_\text{final}\rangle}{\norm{X_t-X_\text{final}}^2}.$$
The aiming condition corresponding to the direction~$D_t$ holds empirically if $\mu_t(D_t)>0$.

Figure~\ref{fig:llama_polar_aiming_mu} shows the absolute values of $\mu(D_t)$ for Muon ($\beta_1=0.95$),
where the solid line depicts $\mu_t$ for $\mu_t>0$ and dotted line plots $|\mu_t|$ when $\mu_t <0$.
In addition to $D_t=\ortho(\hM_t)+\lambda X_t$, we also plot the corresponding values for the gradient $G_t$, the momentum $\hM_t$, and the update $\ortho(\hM_t)$.
The left panel shows that the update direction $\polar(\hM_t)+\lambda X_t$ is positively aligned with $X_t-X_{\text{final}}$ throughout training except for a short period in the beginning. In contrast, $G_t$, $\hM_t$ and $\ortho(\hM_t)$ come into positive alignment with $X_t-X_\text{final}$ in a much later stage, implying the corresponding aiming conditions hold in a much smaller neighborhood of $X_{\text{final}}$.
The right panel shows $|\mu_t|$ for $\ortho(\hM_t)+\lambda X_t$ computed over various parameter groups: matrix parameters updated by Muon, scalar parameters updated by AdamW, attention and MLP blocks at layers $6$, $12$, and $18$, the embedding layer, the language-model head, and the normalization parameters.
It demonstrates that the positive alignment persists when the update direction is restricted to parameter sub-blocks, providing empirical evidence for a stronger, blockwise form of the aiming condition.

Figure~\ref{fig:llama_asgo_aiming_mu} plots empirical $|\mu_t|$ for ASGO $(\beta_1=0.9, \, \beta_2=0.95)$.
We observe that the aiming condition for $D_t=(\hW_t+\epsilon I)^{-1/2}\hM_t+\lambda X_t$ holds empirically throughout the training process for the whole parameter group and most of the sub-blocks.

\begin{figure}[t]
    \centering
    \includegraphics[width=0.98\linewidth]{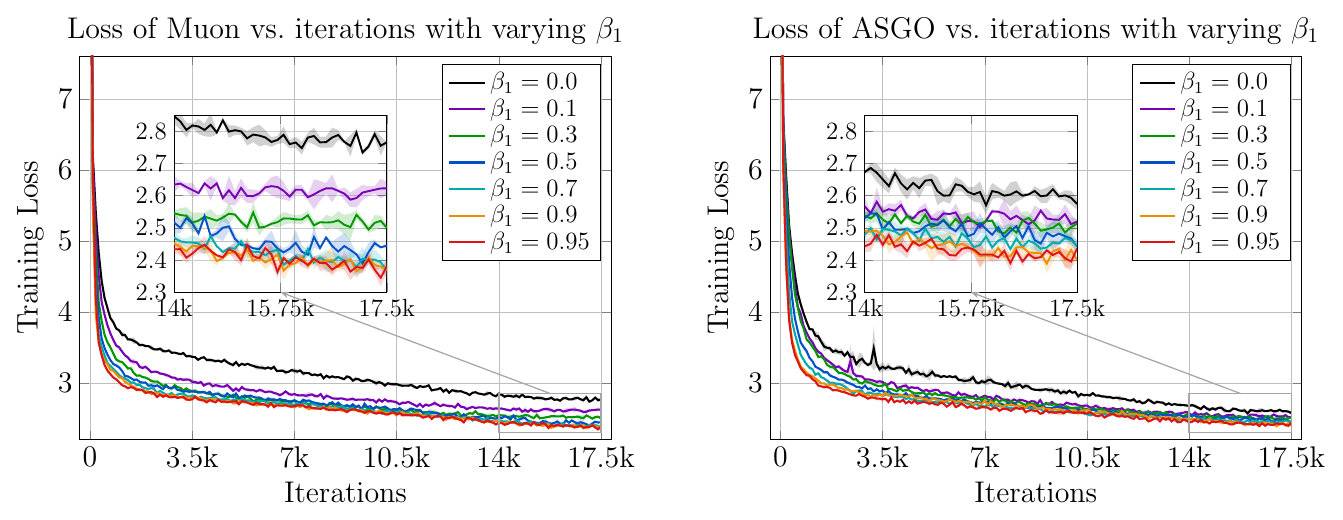}
\vspace{-1ex}
    \caption{Loss curves of Muon and ASGO with different $\beta_1$'s, averaged over 3 random seeds.}
    \label{fig:beta_loss_curve_bs16}
\end{figure}

\subsection{Effect of the smoothing factor}

Our main result in Theorem~\ref{thm:practical-smsa-dwd-almost-sure} states that the algorithms converge almost surely to a neighborhood of some solution~$X_*$, and the radius of the neighborhood is determined by the deviation of the expected practical update from its conceptual counterpart.
Theorems~\ref{thm:overall_polar_err_bound} and~\ref{thm:overall_asgo_err_bound} give concrete bounds on the deviation for Muon and ASGO respectively, both of which depend on the bias and variance of~$\hM_t$ and~$\hW_t$.
We note that the result on Muon does not have explicit dependence on~$\hW_t$, as a result of our simplification under the assumption that the signal fraction matrix~$S_t$ is approximately diagonal (see Section~\ref{sec:derive-muon}).
This is a good approximation, for example, if the noise in~$G_t$ is relatively small (high signal-to-noise ratio).

The instrument for adjusting the bias-variance tradeoff of an EMA estimators is its smoothing factor.
Increasing the smoothing factor in general reduces the variance but potentially increases the bias (depending on the smoothness of the quantity under estimation).
In the second set of exepriments, we examine the performance of Muon and ASGO under different values of the smoothing factor $\beta_1 \in \{0.0, 0.1, 0.3, 0.9, 0.95\}$.
For each~$\beta_1$, we extensively sweep the learning rate, effective weight decay, and the second-moment parameter~$\beta_2$.
The best value for $\beta_2$ is $0.95$ for every $\beta_1$.
After identifying the optimal hyperparameters, we repeat the training experiments across three independent random seeds across model initialization and data loading.

Figure~\ref{fig:beta_loss_curve_bs16} shows that for both Muon and ASGO, the training loss decreases monotonically as~$\beta_1$ increases, with marginal gains diminishing as $\beta_1$ approaches~1.
Since increasing~$\beta_1$ mainly reduces the variance of~$\hM_t$,
this indicates that the dominant factor in the bias-variance tradeoff is the variance of~$\hM_t$; see the bounds
in Theorems~\ref{thm:overall_polar_err_bound} and~\ref{thm:overall_asgo_err_bound}.

Figure~\ref{fig:beta_loss_analysis_bs16} (left) reveals a performance crossover: at $\beta_1 \leq 0.3$, ASGO achieves a lower final loss, whereas at $\beta_1 > 0.5$, Muon outperforms ASGO.
This is consistent with our theory that Muon is a good approximation of the optimal second-moment method under high signal-to-noise ratio, which mostly corresponds to large~$\beta_1$ for EMA estimators.
Arguably ASGO employs a better approximation of the optimal preconditioner across all scenarios as a conceptual method.
But with noisy, practical estimators, it is more prone to bad numerical conditioning caused by forming the Gram matrix $G_t G_t^T$, whose condition number squares that of~$G_t$.
The worse numerical conditioning of ASGO can be observed by comparing the bounds in~\eqref{eq:overall_polar_err_bound} and~\eqref{eq:sqrtinv_err_bd}.

Figure~\ref{fig:beta_loss_analysis_bs16} (right) depicts a microscopic crossover for $\beta_1=0$: Muon decreases the loss faster than ASGO in the beginning of the training, but ASGO catches up and has a lower loss during the rest of the training process. We can also interpret this phenomenon through the lens of signal fractions.
When~$X_t$ is far from~$X_*$, the random regression function~$G_t$ has relatively high signal-to-noise ratio, thus the preconditioner of Muon can be very effective.
As $X_t$ getting closer to $X_*$, the signal-to-noise ratio drops and ASGO becomes more effective in the later stage of training.
This phenomenon persists for small values of $\beta_1$
(Figure~\ref{fig:beta_loss_curve_bs16_b1=0109}, left).
However, with $\beta_1\geq 0.5$, the EMA estimator has significantly smaller variance and we observe that the microscopic crossover disappears
(Figure~\ref{fig:beta_loss_curve_bs16_b1=0109}, right)
and Muon obtains a lower loss than ASGO throughout the training process.

\begin{figure}[t]
    \centering
    \includegraphics[width=0.98\linewidth]{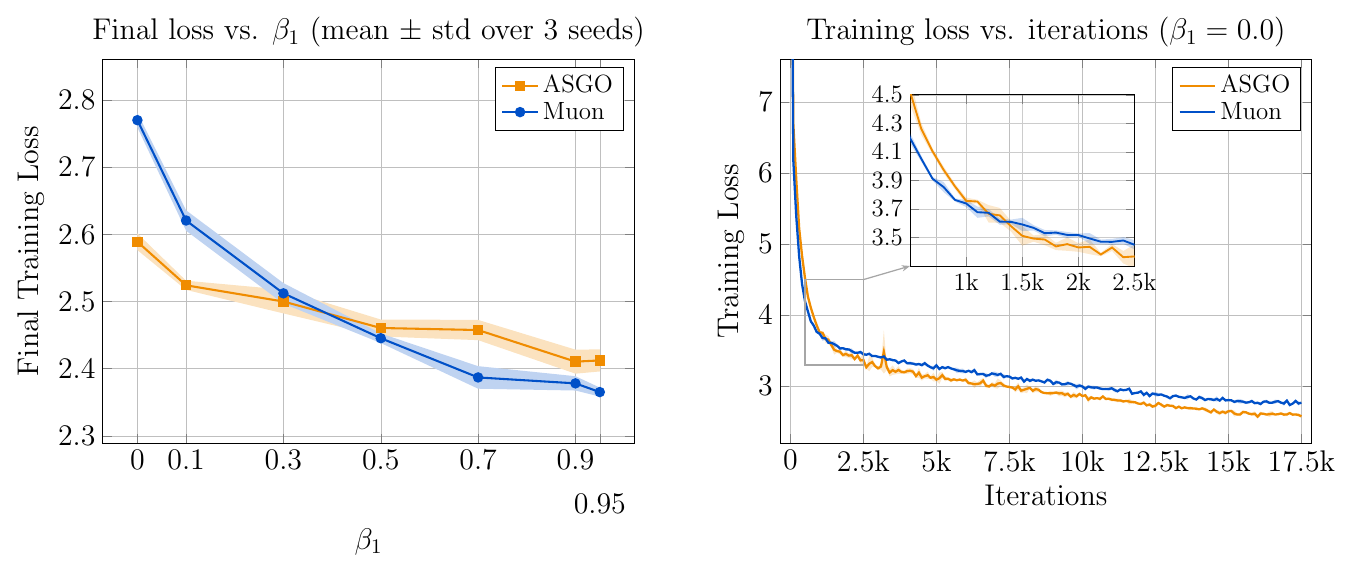}
\vspace{-1ex}
    \caption{Left: Final loss of Muon and ASGO versus different smooth factors $\beta_1$. Right: Loss curves of Muon and ASGO with $\beta_1=0.0$.}
    \label{fig:beta_loss_analysis_bs16}
\end{figure}

\begin{figure}[t]
    \centering
    \includegraphics[width=0.98\linewidth]{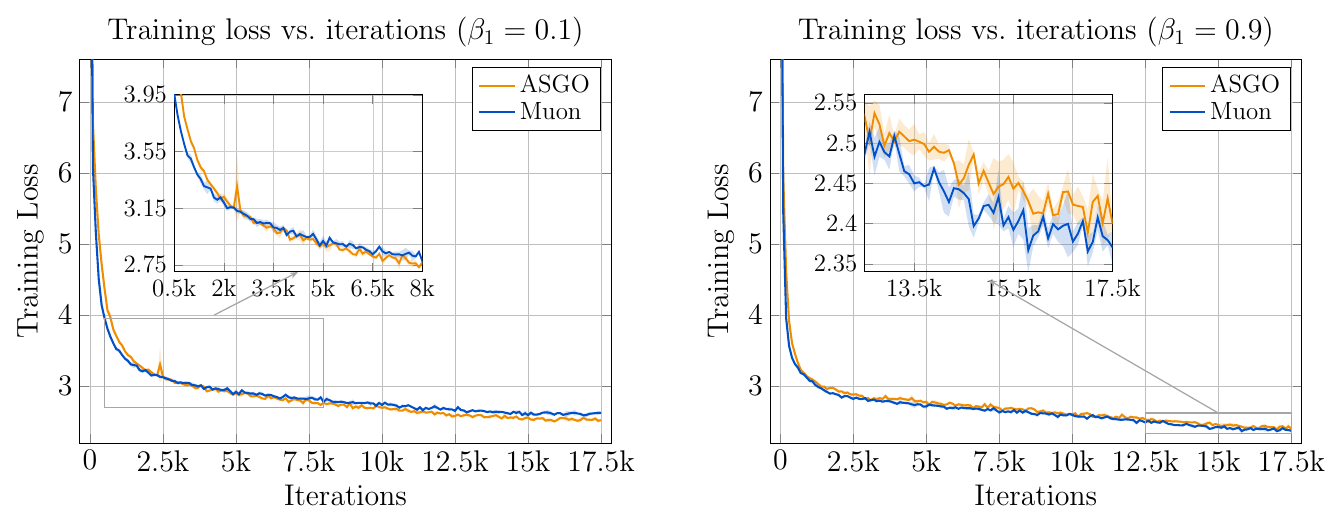}
\vspace{-1ex}
    \caption{Loss curves of Muon and ASGO, averaged over 3 seeds. Left: $\beta_1=0.1$. Right: $\beta_1=0.9$.}
    \label{fig:beta_loss_curve_bs16_b1=0109}
\end{figure}

\section{Conclusion and discussion}
\label{sec:conclusion}

We study a general class of second-moment stochastic approximation (SA) methods. They are derived from a common framework of optimal preconditioning for solving stochastic matrix equations, and different matrix products and simplifications lead to popular algorithms such as Adam, Muon, and their many variants.
A unified approach of convergence analysis was developed based on a two-stage procedure: the first stage focuses on the analysis of conceptual methods that work with exact first and second moments; the second stage establishes convergence of practical methods by bounding the their deviation from the conceptual counterparts.
The deviation bounds depend on the bias and variance of the first- and second-moment estimators without specifying their exact forms, which provides great flexibility and generality of our framework.

The traditional approach of preconditioning for solving nonlinear system of equations is bases on computing or approximating the Jacobian (corresponding to second-order methods in optimization).
Many recent and concurrent works focus on extending this approach to the stochastic setting in order to analyze modern deep-learning optimizers and develop new ones
\cite[e.g.,][]{deng2018optimaladaptive, orvieto2024ngn, su2025isotropiccurvaturemodelunderstanding, abreu2026potential, kovalev2026sgd}.
In contrast, we introduce a \emph{statistical prospect} based on second-moment preconditioning, which tackles the fundamental challenge of working with randomness in the stochastic setting
(just as second-order derivatives combat anisotropic curvature).
This approach has several advantages:
\begin{itemize}
\item %
It provides a \emph{unified view} of many algorithms of different apparent forms, interpreting them as results of simplifying the same optimal preconditioner based on different assumptions on the structure of the noise (through the signal-fraction matrix).
\item %
A major advantage of second-moment SA methods over first-moment methods is the \emph{scaling-invariant} property inherited from the optimal preconditioner. As a consequence, they work well even if the expected regression function (first-moment) is non-smooth or discontinuous.
\item %
The statistical prospect allows us to leverage many classical results in the stochastic approximation literature to establish \emph{almost sure convergence}, which is essential to support the ever-growing scale and cost of training modern deep learning models.
\end{itemize}

In the following, we discuss some limitations of our approach as well as opportunities for future development.
\begin{itemize}
\item Our result in Theorem~\ref{thm:practical-smsa-dwd-almost-sure} guarantees almost sure convergence to a neighborhood of some~$X_*$ satisfying the aiming condition, not exactly to~$X_*$.
This is a tradeoff with the generality of our two-stage analysis framework, which does not specify particular first- and second-moment estimators.
We give an example in Appendix~\ref{sec:persistent-bias-example} illustrating that a persistent bias can prevent convergence to~$X_*$ itself.
On the other hand, the deviation bounds in~\eqref{eq:overall_polar_err_bound} and~\eqref{eq:sqrtinv_err_bd} indicates that the iterates may converge arbitrarily close to~$X_*$ if the estimators have diminishing bias and variance. Even with persistent bias, it is possible to specify nuanced aiming conditions for particular estimators that lead to almost convergence to a singleton.

\item The aiming condition~\eqref{eqn:aiming-smsa-dwd} is in the weak form of a variational inequality \cite[e.g.,][]{Kinderlehrer1980,FacchineiPang2003}.
There is a large literature on efficient algorithms for solving variational inequalities, both in the deterministic setting
\cite[e.g.,][]{Korpelevich1976extragradient,Nemirovski2004prox,Nesterov2007Dual,Malitsky2015} and the stochastic setting
\cite[e.g.,][]{juditsky2011stochastic,chen2017accelerated,iusem2017extragradient,Yousefian2017Smoothing,Kotsalis2022variational1,JiLanZhu2026}.
Many of them may be applied in the context of this paper, potentially obtaining fast convergence rates together with second-moment preconditioners.

\item With $\lambda>0$, the target point~$X_*$ in Assumption~\ref{assum:aiming-smsa-dwd} is a solution to the \emph{regularized} variational inequality~\eqref{eqn:aiming-smsa-dwd}, which may no longer satisfy the original matrix equation $M(X)=0$.
In the context of machine learning, the regularized solution may have better generalization performance \cite{vapnik1998statistical}.
However, with unlimited amount of data in the modern era of deep learning, regularization becomes insignificant or even harming for generalization, but remains necessary for improving training speed and stability.
Therefore, a sensible strategy is to apply diminishing regularization strength, i.e., replace the constant $\lambda$ with a sequence $\lambda_t\to 0$.
This is consistent with our theory that the aiming condition~\eqref{eqn:aiming-smsa-dwd} requires a sufficiently large $\lambda_t$ when~$X_t$ is far from~$X_*$ (see Remark~\ref{rem:need-regularization}) but can holds without regularization near~$X_*$ (see Example~\ref{ex:matrix-polynomial}).
This strategy has been shown to be very effective for both Adam \cite{Defazio2025AdamC} and Muon \cite{Apte2026}, and an analysis that supports our reasoning here is given in \cite{Apte2026}.

\item
It is highly desirable to combine the second-moment methods proposed in this paper with Jacobian-based preconditioners to combat randomness and curvature simultaneously. In fact, these two aspects are not independent of each other and often intertwined.
For example, second-moment estimators also appear in stochastic Gauss-Newton methods \cite[e.g.,][]{LecunBottou1998,bottou2018optimization}, where they are interpreted as good approximation of the Hessian when the objective gap is small \cite[e.g.,][Chapter~10]{nocedal2006numerical}.
However, in Gauss-Newton methods, the second moment estimators appear in the denominator without the square root. Therefore, they do not possess the scaling invariant property, which is a natural consequence from the perspective of second-moment preconditioning.
The recent work of Newton-Muon method~\cite{DuSu2026NewtonMuon} combines elements of both second-order derivates and second moment.
This is a very promising direction for future research.

\end{itemize}

\appendix

\section{Proof of the extended Dvoretzky's theorem}
\label{sec:appendix_dvoretzky}

Our proof is based on the approach of Derman and Sacks \cite{DermanSacks1959}. We first present a slightly general version of their key lemma on a deterministic recursion of multiple sequences.

\begin{lemma}[{Extension of~\cite[][Lemma~1]{DermanSacks1959}}]
\label{lem:deterministic-max-recursion}
Let $\{v_t\}$, $\{a_t\}$, $\{b_t\}$, $\{c_t\}$, and $\{h_t\}$ be
sequences of nonnegative real numbers such that
\begin{equation}\label{eq:deterministic_recursion_coeff}
a_\infty:=\limsup_{t\to\infty}a_t<\infty,
\qquad
\sum_{t=0}^\infty b_t<\infty,
\qquad
\sum_{t=0}^\infty c_t<\infty,
\qquad
\sum_{t=0}^\infty h_t=\infty.
\end{equation}
Suppose that there is a nonnegative integer $T_0$ such that for all $t \geq T_0$,
\begin{equation}\label{eq:deterministic_recursion}
v_{t+1}
\leq
\max\{a_t,(1+b_t)v_t+c_t-h_t\}.
\end{equation}
Then
\[
\limsup_{t\to\infty}v_t\leq a_\infty.
\]
\end{lemma}

\label{sec:appendix_proof_dvoretzky}
\begin{proof}[\textbf{Proof of Theorem~\ref{thm:Dvoretzky}}]
In this proof, we write $T_t$ for $T_t(X_0,\ldots,X_t)$. The proof converts
the vector martingale cross term into a scalar martingale difference and then applies Lemma~\ref{lem:deterministic-max-recursion} to a rescaled squared distance.
Define
\[
w_t:=\frac{2\langle T_t-X_*,Y_t\rangle}
{1+\|T_t-X_*\|_F^2}.
\]
The coefficient of $Y_t$ in $w_t$ is $\mathcal F_t$-measurable and has
Frobenius norm at most one. Hence
\[
\E[w_t\mid\mathcal F_t]=0,
\qquad
\E[w_t^2]\leq\E\|Y_t\|_F^2.
\]
The partial sums of $\sum_t w_t$ therefore form an $L^2$-bounded
martingale, so $\sum_t w_t$ converges almost surely. Moreover, Tonelli's
theorem and~\eqref{eqn:yt-squared-summable} give
\[
\sum_{t=0}^\infty\|Y_t\|_F^2<\infty
\quad\text{almost surely}.
\]
Since $|w_t|\leq\|Y_t\|_F$, this also gives
$\sum_t w_t^2<\infty$ almost surely.

Intersect these two probability-one events with the common
probability-one event in the statement of the theorem, and fix an arbitrary
sample point in the resulting event. From this point onward, all random
matrices and coefficients denote their values at this fixed sample point.
Thus they are ordinary numerical sequences satisfying every displayed
recursion, inequality, limit, and summability condition used below.

Since $w_t\to0$ and $h_t$ is eventually nonnegative, choose an integer
$N\geq t_0$ such that $|w_t|\leq1/2$ and $h_t\geq0$ for every $t\geq N$.
Define
\[
p_t:=\prod_{j=N}^{t-1}(1+w_j),\qquad t\geq N,
\]
where the empty product gives $p_N=1$.
Every $p_t$ is positive. Because $\sum_t w_t$ converges and
$\sum_t w_t^2<\infty$, the expansion
$\log(1+w_t)=w_t+O(w_t^2)$ shows that
\[
p_t\longrightarrow p_\infty
\quad\text{for some }p_\infty\in(0,\infty).
\]

The definition of $w_t$ gives the exact identity
\[
1+\|X_{t+1}-X_*\|_F^2
=\bigl(1+\|T_t-X_*\|_F^2\bigr)(1+w_t)+\|Y_t\|_F^2.
\]
Using $p_{t+1}=p_t(1+w_t)$ and
the recursive upper bound~\eqref{eqn:Dvoretzky-ineq}, we obtain
\begin{align*}
\frac{1+\|X_{t+1}-X_*\|_F^2}{p_{t+1}}
&=\frac{1+\|T_t-X_*\|_F^2}{p_t}
+\frac{\|Y_t\|_F^2}{p_{t+1}}\\
&\leq\max\Biggl\{
\frac{1+a_t}{p_t}+\frac{\|Y_t\|_F^2}{p_{t+1}},\,
(1+b_t)\frac{1+\|X_t-X_*\|_F^2}{p_t}
+\frac{c_t}{p_t}+\frac{\|Y_t\|_F^2}{p_{t+1}}
-\frac{h_t}{p_t}-\frac{b_t}{p_t}
\Biggr\}\\
&\leq\max\Biggl\{
\frac{1+a_t}{p_t}+\frac{\|Y_t\|_F^2}{p_{t+1}},\,
(1+b_t)\frac{1+\|X_t-X_*\|_F^2}{p_t}
+\frac{c_t}{p_t}+\frac{\|Y_t\|_F^2}{p_{t+1}}
-\frac{h_t}{p_t}
\Biggr\}.
\end{align*}

We now verify the hypotheses of
Lemma~\ref{lem:deterministic-max-recursion}. The state sequence
$\{(1+\|X_t-X_*\|_F^2)/p_t\}_{t\geq N}$ and all the coefficients in the last
display are nonnegative. Since
$p_t\to p_\infty\in(0,\infty)$ and $\sum_t\|Y_t\|_F^2<\infty$,
\[
\limsup_{t\to\infty}\left(
\frac{1+a_t}{p_t}+\frac{\|Y_t\|_F^2}{p_{t+1}}
\right)
\leq\frac{1+a_\infty}{p_\infty}.
\]
The multiplicative sequence remains $b_t$, while
\[
\sum_{t=N}^\infty\left(
\frac{c_t}{p_t}+\frac{\|Y_t\|_F^2}{p_{t+1}}
\right)<\infty,
\qquad
\sum_{t=N}^\infty\frac{h_t}{p_t}=\infty.
\]
The lemma therefore yields
\[
\limsup_{t\to\infty}
\frac{1+\|X_t-X_*\|_F^2}{p_t}
\leq\frac{1+a_\infty}{p_\infty}.
\]
Finally, multiplying by $p_t\to p_\infty$ and subtracting one gives
\[
\limsup_{t\to\infty}\|X_t-X_*\|_F^2
\leq a_\infty.
\]
The fixed sample point was arbitrary in an event of probability one, which
proves the theorem.
\end{proof}

\section{An example with persistent estimator bias}
\label{sec:persistent-bias-example}

We give a simple example to demonstrate that the assumptions of Theorem~\ref{thm:practical-smsa-dwd-almost-sure} do not
guarantee $X_t\to X_*$ when the practical update has a persistent bias.

Consider the following one-dimensional case with
\[
    X_*:=0, \qquad M(X) := X, \qquad \hM_t := M_t + \epsilon, \qquad \phi(M,W):=M,\qquad \alpha_t := \frac{1}{t+2}.
\]
Here $\epsilon>0$ is constant. Take $X_0=0$ and $\lambda=0$, and let all variables be deterministic. We now verify the assumptions in Theorem~\ref{thm:practical-smsa-dwd-almost-sure}. First, the aiming and growth Assumption~\ref{assum:aiming-smsa-dwd} hold with $\mu=1, L=1, C=0$. The error bound Assumption~\ref{assum:bounded_err_smsa} is true with $\epsilon_t=\epsilon$. Indeed, $\E_t\norm{X_0}_F^2<\infty$ is true because we initialize $X_0$ at $0$. The required conditions are:
\begin{enumerate}[label=(\alph*)]
    \item Bounded variance under $\phi$:
    The conditional variance is expressed as:
    \begin{equation*}
        \E_t \bigl [\bigl\|\phi\bigl(\hM_t,\hW_t\bigr) - \E_t \bigl [\phi\bigl(\hM_t,\hW_t\bigr) \bigr ]\bigr \|_F^2 \bigr ] = 0,
    \end{equation*}
    satisfying~\eqref{eq:estimator-smsa-dwd_bd_var} in condition~(a) with $b=0$.

    \item The step sizes are deterministic and satisfy
    \begin{equation*}
        \frac{1}{t+2} \geq 0, \quad \forall t \geq 0, \qquad \sum_{t=0}^{\infty} \frac{1}{t+2}  = \infty, \qquad \sum_{t=0}^{\infty}\frac{1}{(t+2)^2} < \infty.
    \end{equation*}
\end{enumerate}
Now we show that $X_t$ does not converge to $X_*$. For $t\geq1$, the chosen
step sizes and practical update rule give
\begin{align*}
    X_t+\epsilon
    &= X_{t-1}-\alpha_{t-1}(X_{t-1}+\epsilon) + \epsilon\\
    &=(1-\alpha_{t-1})(X_{t-1}+\epsilon)\\
    &=\prod_{s=0}^{t-1}(1-\alpha_s)(X_0+\epsilon).
\end{align*}
By construction,
\[
    \prod_{s=0}^{t-1}(1-\alpha_s) = \prod_{s=0}^{t-1}\frac{s+1}{s+2} = \frac{1}{t+1}\to 0.
\]
Hence $X_t\to -\epsilon\neq0=X_*$. Exact convergence fails even though all
stated assumptions hold.

\section{Perturbation bound for matrix square-root inverse}
\label{sec:sqrtinv-perturb-bound-proof}
\begin{proof}[Proof of Lemma~\ref{lem:mat_invsqrt_uniform_perturbation_bound}]
We use the following integral representation of the matrix square-root inverse function \cite[][Chapter~5]{higham2008}
    \[
    A^{-1/2} = \frac{1}{\pi} \int_{s=0}^{\infty} s^{-1/2} (A+sI)^{-1} ds.
    \]
For two symmetric positive matrices $A$ and $B$, we have
    \begin{align*}
        A^{-1/2} - B^{-1/2} &= \frac{1}{\pi} \int_{s=0}^{\infty} s^{-1/2} \bigl((A+sI)^{-1} -  (B+sI)^{-1}\bigr) ds\\
        &= \frac{1}{\pi} \int_{s=0}^{\infty} s^{-1/2} (A+sI)^{-1}(B-A)(B+sI)^{-1} ds,
    \end{align*}
where the second line follows from the resolvent identity.
Using the sub-multiplicative property of matrix norms, we get:
    \begin{align*}
        \bigl\|A^{-1/2} - B^{-1/2}\bigr\|_F
        \leq \frac{1}{\pi} \int_{s=0}^{\infty} s^{-1/2} \bigl\|(A+sI)^{-1}\bigr\|_2 \bigl\|B-A\bigr\|_F \bigl\|(B+sI)^{-1}\bigr\|_2 ds.
    \end{align*}
Since $A \succeq \lambda_A I$ and $B \succeq \lambda_B I$, we can derive bounds for the two operator norms:
    \[
    \|(A+sI)^{-1}\|_2 \leq \frac{1}{\lambda_A+s},
    \qquad
    \|(B+sI)^{-1}\|_2 \leq \frac{1}{\lambda_B+s}.
    \]
    Substituting these results into the previous upper bound and evaluating the integral, we get
    \begin{align*}
        \bigl\|A^{-1/2} - B^{-1/2}\bigr\|_F
        &\leq \frac{\|B-A\|_F}{\pi}
        \int_{s=0}^{\infty}s^{-1/2}(\lambda_A+s)^{-1}
        (\lambda_B+s)^{-1}\,ds\\
        &=\bigl(\lambda_A\lambda_B^{1/2}
        +\lambda_B\lambda_A^{1/2}\bigr)^{-1}\|B-A\|_F,
    \end{align*}
    as desired.
\end{proof}

\bibliographystyle{unsrt}
\bibliography{ref}
\end{document}